\documentclass[leqno,a4paper,12pt]{amsart}
\usepackage{amsmath,amssymb,amsthm}
\usepackage{amsmath,amssymb,amsthm,hyperref}
\usepackage[pagewise]{lineno}%\linenumbers
\usepackage{graphicx}
\usepackage{version} % include, exclude and comment
\usepackage{color}
\usepackage{pstricks-add}
\usepackage[bf]{caption}     % captionstyles
\usepackage{psfrag}          % replace text in figures
\usepackage{geometry}        % for good layout
\usepackage{bm}  % bold symbols in math mode \boldsymbol{\alpha}
\usepackage[utf8]{inputenc}
\usepackage{setspace}
\usepackage{wrapfig}
\usepackage{epstopdf}
\usepackage[mathscr]{eucal}
\usepackage{enumerate}
\usepackage{bm} %\bm{a} bold in math mode everything
 \newtheorem{theorem}{Theorem}

 \newtheorem{lemma}[theorem]{Lemma}
 \newtheorem{proposition}[theorem]{Proposition}
 
 {\theoremstyle{definition}
 \newtheorem{definition}[theorem]{Definition}
 \newtheorem{remark}[theorem]{Remark}

 }
\DeclareMathOperator{\disthau}{dist_{Hau}}

\def\N{\ensuremath{\mathbb N}} %natural numbers
\newcommand{\sgn}{\operatorname{sgn}}
\newcommand{\DDD}{\Delta}

\newcommand{\LLL}{\Lambda}

\newcommand{\aaa}{\alpha}

\newcommand{\bbb}{\beta}

\newcommand{\cak}{{\mathcal K}}

\newcommand{\ddd}{\delta}
\newcommand{\dd}{\partial}

\newcommand{\ds}{\displaystyle}
\newcommand{\eee}{\varepsilon}

\newcommand{\mmm}{\mu}
\newcommand{\mm}{{\mathfrak m}}

\newcommand{\MM}{\mathfrak{M}}

\newcommand{\oo}{\infty}

\newcommand{\seq}{\succeq}

\newcommand{\sse}{\subset}
\newcommand{\sss}{\sigma}

\newcommand{\TTT}{\Theta}

\newcommand{\uM}{\underline{M}}

\newcommand{\xx}{\times}

\newcommand{\omm}{\overline{m}}

\newcommand{\saab}{_{\aaa,\bbb}}

\renewcommand{\ggg}{\gamma}

\renewcommand{\lll}{\lambda}

\begin{document}
\title{Isentropes and Lyapunov exponents}
\author[Z. Buczolich\ and \ G. Keszthelyi]{Zolt\'an Buczolich*  and  Gabriella Keszthelyi**}
\newcommand{\acr}{\newline\indent}
\address{\llap{*\,}Department of Analysis\acr
ELTE E\"otv\"os Lor\'and University\acr
P\'azm\'any P\'eter S\'et\'any 1/c, 1117 Budapest, Hungary\acr
ORCID ID: 0000-0001-5481-8797}
\email{ buczo@caesar.elte.hu}
\urladdr{http://buczo.web.elte.hu}
\address{\llap{**\,}Department of Analysis\acr
ELTE E\"otv\"os Lor\'and University\acr
P\'azm\'any P\'eter S\'et\'any 1/c, 1117 Budapest, Hungary\acr}
\email{keszthelyig@gmail.com}

\thanks{The first listed
author  was supported by the Hungarian
National Foundation for Scientific Research Grant 124003.
During the preparation of this paper this author was a visiting researcher
at the R\'enyi Institute.
\newline\indent The second listed
author  was supported by the Hungarian
National Foundation for Scientific Research Grant 124749.
}

\begin{abstract}
We consider skew tent maps $T_{ { \alpha }, { \beta }}(x)$ such that  $(\aaa,\bbb)\in[0,1]^{2}$ is the turning point of $T\saab$, that is, $T_{ { \alpha }, { \beta }}=\frac{ { \beta }}{ { \alpha }}x$ for $0\leq x \leq { \alpha }$ and $T_{ { \alpha }, { \beta }}(x)=\frac{ { \beta }}{1- { \alpha }}(1-x)$
for $ { \alpha }<x\leq 1$. 
%The dynamics of $T_{ { \alpha }, { \beta }}$ for $( { \alpha }, { \beta })\in [0,1]^{2}$
%is interesting when $( { \alpha }, { \beta })\in U=\{ ( { \alpha }, { \beta }): 0.5< { \beta }\leq 1, \  1- { \beta }< { \alpha }< { \beta } \}$.  
 We denote by $\uM=K(\aaa,\bbb)$ the kneading sequence of
 $T\saab$, by $h(\aaa,\bbb)$ its topological entropy and 
 $\Lambda=\Lambda_{\alpha,\beta}$ denotes its Lyapunov exponent.
 For a given kneading squence $\uM$ we consider isentropes (or equi-topological entropy, or equi-kneading curves),
 $(\aaa,\Psi_{\uM}(\aaa))$  such that $K(\aaa,\Psi_{\uM}(\aaa))=\uM$. On these curves the topological entropy $h(\aaa,\Psi_{\uM}(\aaa))$ is constant.\\
 We show that  $\Psi_{\uM}'(\aaa)$ exists and the 
 Lyapunov exponent $\Lambda_{\alpha,\beta}$ can be expressed by using  the slope of the tangent to the isentrope.
Since this latter can be computed by considering partial derivatives of an auxiliary function $\TTT_{\uM}$,
a series depending on the kneading sequence
which converges at an exponential rate, this provides an efficient new
method of finding the value of the Lyapunov exponent of these maps.  
\end{abstract}

\maketitle

\section{Introduction}
Consider a point $(\alpha,\beta)$ in the unit square $[0,1]^2$.
 Denote by $T_{\alpha,\beta}(x)$ the skew tent map.
\begin{equation}\label{021602a}
T_{\alpha,\beta}(x) =\left\{
\begin{array}{clcr}
			 L_{\alpha,\beta}(x)=\frac{\beta}{\alpha} x & {\rm if}    & 0 \leq x \leq \alpha,      \\
			  R_{\alpha,\beta}(x)=\frac{\beta}{1-\alpha}(1-x)  &  {\rm if}  & \alpha< x \leq 1.
	\end{array}
\right.
\end{equation}
 To avoid trivial dynamics we suppose that $0.5< \beta\leq 1$ and
$\alpha \in (1-\beta,\beta)$. We denote by $U$ the region of $[0,1]^2$ consisting of these $[\alpha,\beta]$.  
We denote by $\uM=K(\aaa,\bbb)$ the kneading sequence of
 $T\saab$, by $h(\aaa,\bbb)$ its topological entropy and by $\Lambda=\Lambda_{\alpha,\beta}$ denotes its Lyapunov exponent.
 The set of all possible kneading sequences is denoted by $\MM=\{ K(\aaa,\bbb):
(\aaa,\bbb)\in U \}.$
 For a given kneading squence $\uM$ we consider isentropes (or equi-topological entropy, or equi-kneading curves)  $(\aaa,\Psi_{\uM}(\aaa))\in U$  such that $K(\aaa,\Psi_{\uM}(\aaa))=\uM$. On these curves the topological entropy $h(\aaa,\Psi_{\uM}(\aaa))$ is constant.
On Figure \ref{isererint1} on the left half $T_{.3,.8}$ is considered. On the bottom 
part of the figure one can see the first few entries of the kneading sequence.
To visualize the isentrope the computer plotted in black some pixels
which correspond to parameter values with similar initial segment of kneading sequence. To obtain a not too thick region the length of this initial segment
depends on the parameter region. For example on the left half of Figure \ref{isererint2} there is a thicker region, which can be made thinner by considering
longer initial segments. However if the initial segment is too long, the computer is not finding enough pixels from the given equi-kneading region, see for example
the right half of Figure \ref{isererint1} where close to the upper left corner of the unit square the plot is too thin. 
 
\begin{figure}[ht]
\includegraphics[width=1\textwidth]{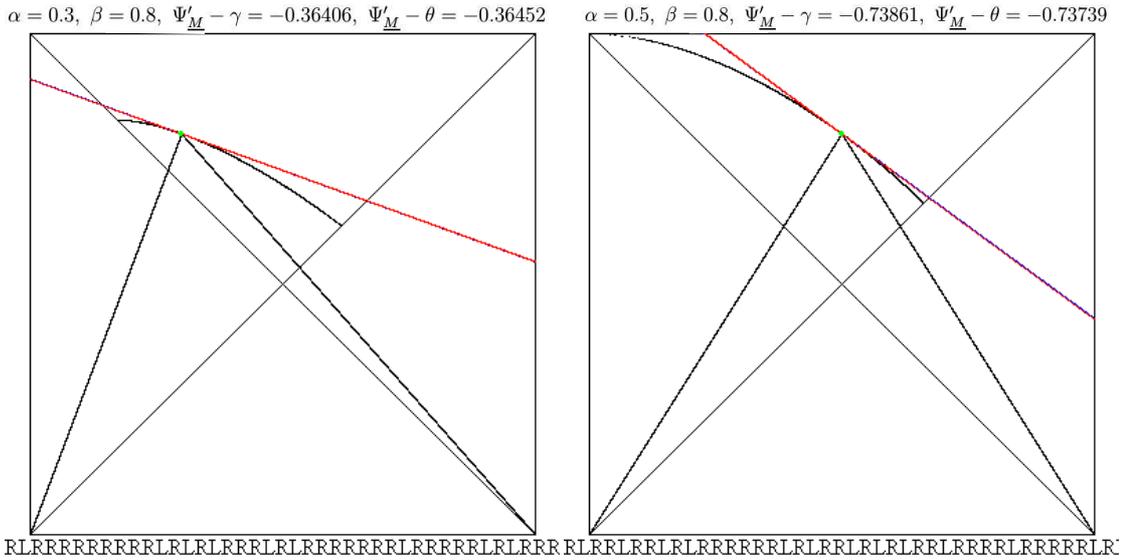}
\caption{Tangents to isentropes computed from $\ggg$ and from $\Theta$\label{isererint1}}
\end{figure}

We will see in this paper that the isentropes $(\aaa,\Psi_{\uM}(\aaa))$  are 
continuously differentiable curves namely we prove the following theorem. 
\begin{theorem}\label{thdiffnonmark}
If $\underline{M} \in \mathfrak{M}$ then $\Psi'_{\underline{M}}$ exists and is continuous on its domain, $(\alpha_1(\underline{M}),\alpha_2(\underline{M}))$.
\end{theorem}

In \cite{[BalSma2]} one can find results and general methods concerning isentropes 
and smoothness of isentropes
for different type perturbations of piecewise expanding unimodal maps.
It is a significant difference that in \cite{[BalSma2]} the horizontal coordinate of the turning point of the maps is fixed during these perturbations.

What we found really interesting that the derivatives of these curves can be used
to compute the Lyapunov exponents of the skew tent maps $T_{\aaa,\bbb}$.
This is the main result of this paper:

\begin{theorem}\label{thlamark}
Suppose $(\alpha_0,\beta_0) \in  U$, 
 $\Lambda = \Lambda_{\alpha_0,\beta_0}$ denotes the Lyapunov exponent of $T_{\alpha_0,\beta_0}$ and $(\alpha, \Psi_{\underline{M}}(\alpha))$ is the isentrope satisfying $\beta_0=\Psi_{\underline{M}}(\alpha_0)$.
 Then we have the following formula
\begin{equation}\label{IL1*a}
\Lambda_{\alpha_0,\beta_0}=\Lambda=\gamma \log \frac{\beta_0}{\alpha_0}+(1-\gamma)\log\frac{\beta_0}{1-\alpha_0}, \text{ where } \gamma \text{ satisfies }
\end{equation}
\begin{equation}\label{IL1*b}
\gamma=\frac{\frac{\Psi'_{\underline{M}}(\alpha_0)}{\beta_0}+\frac{1}{1-\alpha_0}}{\frac{1}{\alpha_0}+\frac{1}{1-\alpha_0}}=\alpha_0(1-\alpha_0)\frac{\Psi'_{\underline{M}}(\alpha_0)}{\beta_0}+\alpha_0.
\end{equation}
Moreover, if $\mu$ denotes the acim (absolutely continuous invariant measure) of $T_{\alpha_0,\beta_0}$ then
\begin{equation}\label{ILE2*a}
\gamma=\mu([0,\alpha_0]).
\end{equation}
\end{theorem}

%To compute Lyapunov exponents is not easy.

To study equi-topological entropy, or equi-kneading curves in the region $U$ 
in \cite{[STQ]}
we introduced
the auxiliary functions $\TTT_{\uM}$.
Suppose that we have a  given kneading-sequence $\uM$ and 
\begin{equation}\label{041801aa}
\underline{M}^{-}=R\underbrace{L\dots L}_{m_1}R\underbrace{L\dots L}_{m_2}R\underbrace{L\dots L}_{m_3}R \dots .
\end{equation}
Here $\uM=\underline{M}^{-}$ if the turning point is not periodic, that is
$T^k_{\aaa,\bbb}(\aaa)\not=\aaa$ for $k\in\N$. In this case there is no $C\in \uM$.
The set of such kneading sequences is denoted by $\MM_{\oo}$.
The cases when the truning point is periodic, that is when $C$ appears in
$\uM$ will play a very important role in this paper.
 The set of these kneading sequences is denoted by $\MM_{<\oo}$.
These are the ones ending with $C$. In this case $\underline{M}^{-}$
can be defined in many ways. One such way was discussed in \cite{[STQ]}.
However, for our definition of suitable $\TTT_{\uM}$ functions any of the following definitions can be used. 
Concatenate $\uM$ with itself infinitely many times. Then
in the right infinite (right) periodic sequence replace the $C$s in an arbitrary manner with $R$s and $L$s. 

For example in our computer simulations each $C$
was replaced by an $L$. This is due to the fact that if $T^k_{\aaa,\bbb}(\aaa)=\aaa$ then $T^{k+1}_{\aaa,\bbb}(\aaa)=\bbb=L_{\aaa,\bbb}(T^k_{\aaa,\bbb}(\aaa))=R_{\aaa,\bbb}(T^k_{\aaa,\bbb}(\aaa))$, that is both the left- 
and right- ``half definitions" of $T^k_{\aaa,\bbb}$ can be used in this case.

\begin{table}
\begin{tabular}{|l|l|l|l|l|}
\hline
$\alpha$&$\beta$&$\gamma$&$\Psi_{\underline{M}}'$--$\gamma$&$\Psi_{\underline{M}}'$--$\Theta$\\ \hline
.3&.8&.20444&-.36406&-.36452\\ \hline
.49&.56&.30996&-.40344&-.4244\\ \hline
.5&.7&.27034&-.64303&-.64064\\ \hline
.5&.8&.26918&-.73861&-.73739\\ \hline
.6&.75&.35597&-.76258&-.76132\\ \hline
.6&.9&.47736&-.4599&-.45991\\ \hline
\end{tabular}
\caption{Tangents calculated from $\Theta$ and $\gamma$\label{tab1}}
\end{table}

We put $\omm_{k}=m_1+m_2+\dots+m_k$ with $m_{i}$ defined in \eqref{041801aa} and
\begin{equation}\label{02261a}
  \Theta_{\underline{M}}(\alpha,\beta)=1-\beta+\sum_{k=1}^{\infty}(-1)^{k}\left(\frac{1-\alpha}{\beta}\right)^k\left(\frac{\alpha}{\beta}\right)^{{\overline m}_k}.
\end{equation}
In \cite{[STQ]} we showed
that for $(\aaa,\bbb)\in U$ it follows from $K(\aaa,\bbb)=\uM$ that
$\TTT_{\uM}(\aaa,\bbb)=0$. This means that the equi-topological entropy curve
$\{ (\aaa,\bbb)\in U : K(\aaa,\bbb)=\uM \}$ is a subset of
$\{ (\aaa,\bbb)\in U: \TTT_{\uM}(\aaa,\bbb)=0 \}$, the zero level set of $\TTT_{\uM}$.  
This means that the isentrope $(\aaa,\Psi_{\uM}(\aaa)) $ satisfies
the implicit equation $\TTT_{\uM}(\aaa,\Psi_{\uM}(\aaa))=0$.
By implicit differentiation 
\begin{equation}\label{*implder}
\Psi_{\uM}'(\aaa)=-\frac{\dd_{1}\TTT_{\uM}(\aaa,\Psi_{\uM}(\aaa))}{\dd_{2}\TTT_{\uM}(\aaa,\Psi_{\uM}(\aaa))},
\end{equation}
provided that $\dd_{2}\TTT_{\uM}(\aaa,\Psi_{\uM}(\aaa))\not = 0.$
Since the series in \eqref{02261a} converges at an exponential rate if we consider the partial derivatives we also obtain an exponential convergence rate for the partial
derivatives and hence it is very easy to compute/approximate $\Psi_{\uM}'(\aaa)$
by using \eqref{*implder}. On Figures \ref{isererint1}, \ref{isererint2} and in Table
\ref{tab1} the entries 
%{\tt Psi'-theta}  and  
 $\Psi_{\uM}'-\TTT$
were computed by using this implicit differentiation method by taking into consideration the first $200$ elements of the kneading sequence.

The other approach is to estimate  $\Psi_{\uM}'(\aaa)$ via the Lyapunov exponents.
For the skew tent map $T_{\aaa,\bbb}$, $(\aaa,\bbb)\in U$ there is a unique ergodic {\it acim}
$\mmm_{\aaa,\bbb}=\mmm$, that is a measure absolutely continuous with respect to the Lebesgue measure, $\lll$. Its density $f$ is an invariant function/fixed point of the Frobenius-Perron operator $P_{\alpha,\beta}$, that is $P_{\alpha,\beta}f=f$.
By Birkhoff's ergodic theorem the Lyapunov exponent
\begin{equation}\label{*lyadef}
\LLL_{\aaa,\bbb}=\lim_{N\to \oo}\frac{1}{N}\log |(T_{\aaa,\bbb}^{N})'(x)|=
\lim_{N \rightarrow \infty} \frac{1}{N} \sum_{n=0}^{N-1} \log |T_{\aaa,\bbb}'(T_{\aaa,\bbb}^{n}(x))|
\text{ for }\mmm\text{ a.e. }x.
\end{equation} 
In case of skew tent maps $|T_{\aaa,\bbb}'(x)|=\bbb/\aaa$ if $x<\aaa$ and
$|T_{\aaa,\bbb}'(x)|=\bbb/(1-\aaa)$ if $x>\aaa$ hence if we let 
\begin{equation}\label{IL3*aa}
\gamma= \lim_{N \rightarrow \infty} \frac{1}{N} \sum_{n=0}^{N-1} \chi_{[0,\alpha]}(T^n_{\alpha,\beta}(x))\text{ then }\LLL_{\aaa,\bbb}=\ggg\log\frac{\bbb}{\aaa}+
(1-\ggg)\log \frac{\bbb}{1-\aaa},
\end{equation}
for $\mmm$ a.e. $x$.

\begin{figure}[ht]
\includegraphics[width=1\textwidth]{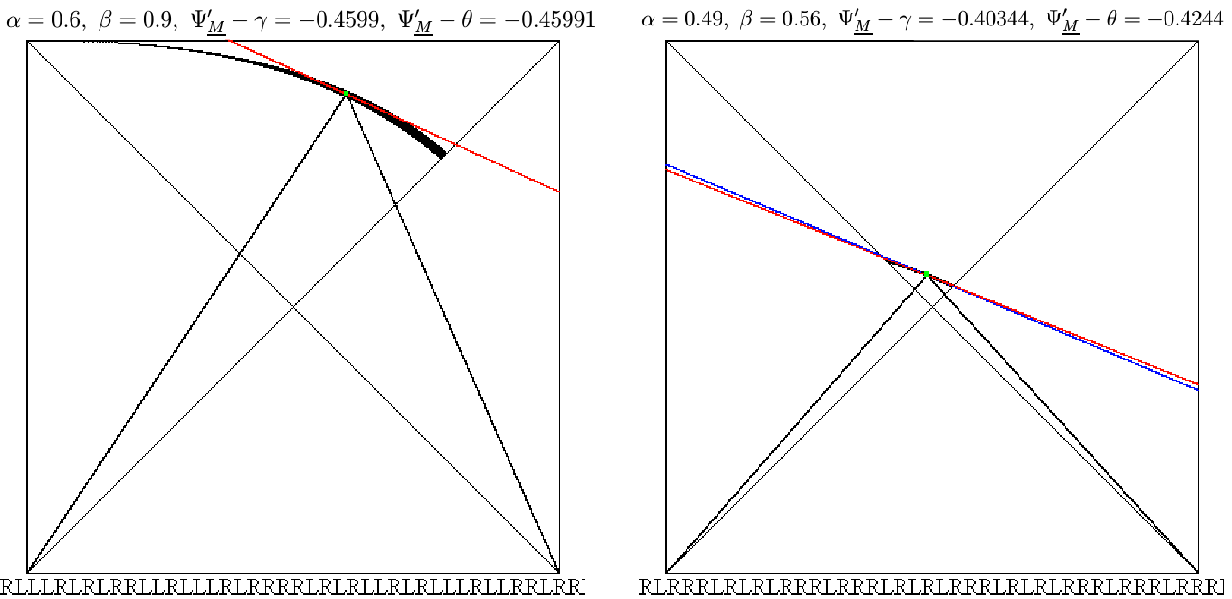}
\caption{More tangents to isentropes computed from $\ggg$ and from $\Theta$\label{isererint2}}
\end{figure}

Hence to estimate the Lyapunov exponent we need to estimate $\ggg$.
This is usually done by using a computer program. For a sufficiently large $N$
and a ``randomly" selected $x$ one computes the sum in \eqref{IL3*aa}.
Actually we have done this as well in our computer simulations. It has turned out that $N=200 000$ was sufficiently large to have a reasonably good estimate for $\ggg$.
In Table \ref{tab1} there is a column $\ggg$ containing these estimates for
the randomly selected parameter values. According to \eqref{IL1*a} of Theorem \ref{thlamark} $\ggg$, and hence $\LLL_{\aaa,\bbb}$ can be expressed by using 
$\Psi'_{\underline{M}}(\alpha)$.
%This is the main result of this paper.
Since $\Psi'_{\underline{M}}(\alpha)$ can be calculated by \eqref{*implder}
using  \eqref{*lyadef}, \eqref{IL3*aa} and \eqref{IL1*a} we can calculate
the Lyapunov exponent for any $T_{\aaa,\bbb}$ with $(\aaa,\bbb)\in U.$
To illustrate the connection between $\Psi'_{\underline{M}}(\alpha)$
and $\LLL_{\aaa,\bbb}$, or $\ggg$ in our computer simulations followed a reverse
approach, using $\ggg$ from \eqref{IL1*a} one can obtain
\begin{equation}\label{IL1*bb}
\Psi'_{\underline{M}}(\alpha)=\frac{(\ggg-\aaa)\bbb}
{\aaa(1-\aaa)} .
\end{equation}
 This means that the computer program calculated an estimate of 
$\ggg$ (and hence of $\LLL$) and this estimate was 
used for calculating the slope of an approximate tangent to the isentrope.
As the images show this method, based on \eqref{IL1*bb}
works, that is the approximate tangents really seem to be tangent
to the isentrope.

In Table \ref{tab1} there is a column  labeled $\Psi_{\underline{M}}'$--$\Theta$
which contains the estimates we obtained for $\Psi'_{\underline{M}}(\alpha)$
by using the estimate for $\ggg$ based on \eqref{IL3*aa}.
As one can see that the estimates we obtained for $\Psi'_{\underline{M}}(\alpha)$
by using the $\TTT_{\uM}$ function in \eqref{*implder} are quite close to the ones 
obtained by using $\ggg$.
On Figures \ref{isererint1} and \ref{isererint2} we plotted both approximate tangents 
to the isentropes, the one calculated from $\ggg$ and the one calculated from $\TTT_{\uM}$. On the color pdf version of the paper the first approximate tangent is in red and the second is in blue. In case only one, the red tangent is visible then
it means that the two approximate tangents are on top of each other. It is also visible that they are indeed "tangent" to the isentrope as well.
On the right half of Figure \ref{isererint2} the two approximate tangents are not exactly on top of each other. This is due to the fact that for the parameter values
$\aaa=0.49$ and $\bbb=0.56$ both $\aaa/\bbb$ and $(1-\aaa)/\bbb$
are close to one and the convergence in the series giving the partial derivatives
of $\TTT_{\uM}$ is slower. To get a better estimate
one needs to consider more than the first $200$
entries of the kneading sequence. On this figure the tiny black region
corresponding to the equi-kneading region is almost completely covered by the blue and red approximate tangents.
We would like to emphasize that our new method based on $\TTT_{\uM}$, even if the number of
iterates is increased from $200$ to a larger number requires still much less many iterates than the other method which needed $1000$ times more iterates
for about the same accuracy.  

Finally, there is one more illustration showing that indeed there is a link between
$\ggg$ and $\Psi'_{\underline{M}}(\alpha)$. On Figure \ref{slycsalad}
the color of pixels in $U$ was calculated based on the first $10$ entries of the kneading sequence. Hence equi-kneading regions containing isentropes are of the same color (modulo screen/pixel resolution). We also plotted three skew tent maps with three different colors and the approximate tangent line computed by using $\ggg$ from \eqref{IL3*aa} substituted into \eqref{IL1*bb}.

%As far as we know in the literature  there were two ways to estimate/approximate Lyapunov exponents of skew tent maps. One method 
%is based on computer programs approximating $\ggg$, or the acim, or its density as we also did in some calculations on our illustrations.
In \cite{[BB]} for the Markov case a histogram of the distribution
of the location in the Markov partition  of the first $50 000$ iterates of a "generic" point
is used to approximate the piecewise constant invariant density function of the acim.
Here again a rather high number of iterates was used.
In \cite{[LC]} a central limit theorem is discussed for the convergence in \eqref{IL3*aa}. 
 The other method, discussed in \cite{[BB]}
is based on the fact that if $K(\aaa,\bbb)\in \MM_{<\oo}$, that is when 
the turning point is 
periodic for $T_{\aaa,\bbb}$ then there is a Markov partition for $T_{\aaa,\bbb}$.
Based on the Markov partition one can  obtain a system of linear equations
and the solution of this system gives us the invariant density function
$f_{\aaa,\bbb}$ of the acim $\mmm_{\aaa,\bbb}$ of $T_{\aaa,\bbb}$.
Then $\ggg=\mmm_{\aaa,\bbb}([0,\aaa])$. (In \cite{[BB]} a different parametrization and notation was used, but we translated it to our notation.)
The drawback of this calculation is that the number of equations is the number 
of elements in the Markov partition. If $K(\aaa,\bbb)\in \MM_{\oo}$ then there is 
no Markov partition, but isentropes corresponding to skew tent maps with Markov
partition are dense in $U$. It was remarked in \cite{[BB]} that in this case
we can also approximate the invariant density by invariant densities of Markov skew tent maps.  In this case the number 
of elements in the Markov partition of these appproximating maps tends to infinity,
making it more and more difficult to solve the system of linear equations.
It also seems for us that Theorem 10.3.2 from \cite{[BG]} was used in an incorrect way in \cite{[BB]}. By this we mean,
 that the way these Markov skew tent maps are approximating the non-Markov one is not satisfying the exact assumptions
of Theorem 10.3.2 in \cite{[BG]}. Since in our paper we also need  approximations
 of skew tent maps by other ones in Proposition \ref{1032} we clarify the way these approximations
work.
For some specific Markov parameter values in \cite{[MaKa]} a central limit behavior is discussed. 

We thank the referee for pointing out that based on methods of \cite{[BalSma1]} see equations (21) and (35) of  \cite{[BalSma1]} another efficient method for estimating the invariant density can be obtained.

Properties of isentropes, especially connectedness in different families of dynamical systems were also studied for example in 
 \cite{[BvS]}, \cite{[MTr]}  and \cite{[Rad]}.

\begin{figure}[ht]
\includegraphics[width=0.7\textwidth]{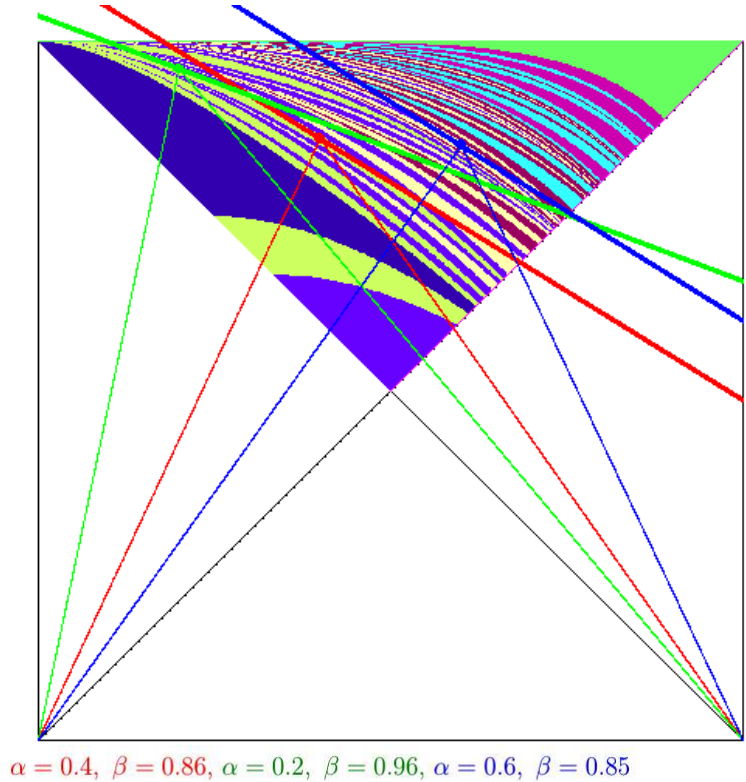}
\caption{Isentropes and tangents computed from $\ggg$\label{slycsalad}}
\end{figure}

This paper is organized the following way. In Section \ref{secprelim} we recall some definitions and results concerning skew tent maps and invariant densities.
In Section \ref{ACIM} we continue to discuss some known results about
absolutely continuous invariant measures and prove Proposition \ref{1032}
which will be the key lemma about approximations of skew tent maps by other ones.
This section concludes with some remarks about uniform Lipschitz properties of isentropes. 

The most involved part of the paper is Section \ref{secismar} in which we prove
Proposition \ref{prlanmark}. This is a special version of the main result of the paper about the relationship between Lyapunov exponents and tangents to isentropes. 
In this proposition we suppose that the isentrope is differentiable at the point considered and we also suppose that
we work with a Markov map.
In later sections we aim towards Theorem \ref{thlamark} to use some 
approximation arguments to remove the assumptions about 
differentiability
and Markovness.

In Section \ref{Disentr} by using Proposition \ref{prlanmark} first we show that
isentropes are continuously differentiable for Markov skew-tent maps.
In this argument we use Proposition \ref{1032} and approximations of our skew tent map by other ones with the same topological entropy.
Then by using another approximation argument based on Proposition \ref{1032}
and approximation of non-Markov maps by Markov maps we generalize this result for arbitrary maps.

Finally, in Section \ref{secisnonmar} we prove Theorem \ref{thlamark} which is the main result of our paper. It is again an approximation argument of non-Markov maps by Markov maps. This way we obtain the general version
of Proposition \ref{prlanmark}. 

\section{Preliminaries}\label{secprelim}

Kneading theory was introduced by J. Milnor and W. Thurston in \cite{[MT]}.
For symbolic itineraries and for the kneading sequences we follow the notation of \cite{[CE]}.

Suppose $T=T_{\aaa,\bbb}$ is fixed for an $(\aaa,\bbb)\in U$ and $x\in [0,1]$.
The extended kneading sequence 
$K(\aaa,\bbb)=\uM_{{\mathrm ext}}=(\mm_{1},\mm_{2},...)\in \{ L,R,C \}^{\N}$
is defined as follows.
If $T_{\aaa,\bbb}^{n}(\aaa)<\aaa$ then $\mm_{n}=L$, 
if $T_{\aaa,\bbb}^{n}(\aaa)=\aaa$ then $\mm_{n}=C$,
and if $T_{\aaa,\bbb}^{n}(\aaa)>\aaa$ then $\mm_{n}=R$.
If there is no $C$ in $\uM_{{\mathrm ext}}$ then the kneading sequence $K(\aaa,\bbb)=\uM=\uM_{{\mathrm ext}}$.
If there are $C$s in $\uM_{{\mathrm ext}}$  then the kneading squence $K(\aaa,\bbb)=\uM$ is a finite string which is obtained by stopping at the first $C$ and throwing away
the rest of the infinite string $\uM_{{\mathrm ext}}$.

Following notation of \cite{[MV]} we denote by $\mathfrak{M}$ the class of kneading sequences $K(0.5,\beta), \,\beta \in (0.5,1]$, which is identical to all
 possible kneading sequences of the form $K(\aaa,\bbb)$, $(\aaa,\bbb)\in U$. 

In \cite{[MV]} a different
parametrization of skew tent maps was used. The functions
$$F_{\lambda,\mu}(x) =\left\{
\begin{array}{clcr}
			{1+\lambda}x & {\rm if}   & x \leq 0      \\
			1-\mu x    &  {\rm if}  & x \geq 0
	
		\end{array}
\right.$$
were considered on $\mathbb{R}$. 

A simple calculation shows that if $(\alpha,\beta) \in U$ then $(\lambda(\alpha,\beta),\mu(\alpha,\beta))=(\frac{\beta}{\alpha},\frac{\beta}{1-\alpha})$ belongs to the region $D'=\{(\lambda, \mu):\lambda > 1, \mu>1, \frac{1}{\lambda}+\frac{1}{\mu} \geq 1\}$
this, apart from a boundary segment, coincides with the parameter region
$D=\{(\lambda, \mu):\lambda \geq 1, \mu>1, \frac{1}{\lambda}+\frac{1}{\mu} \geq 1\}$
 considered in \cite{[MV]}.
In \cite{[STQ]} we gave the explicit formula for the linear homeomorphism
showing that
$T_{\aaa,\bbb}$ and $F_{\lambda(\alpha,\beta),\mu(\alpha,\beta)}$ 
are topologically conjugate.
We use the notation $\cak(\lll,\mmm)$ for the kneading sequence
of $F_{\lll,\mmm}$.
In this parametrization $\MM$ corresponds to the kneading sequences of functions $F_{\mu,\mu}$ with $1<\mu \leq 2$.

We denote by $\prec $ the parity lexicographical ordering of kneading sequences, symbolic itineraries, for the details see \cite{[CE]}.

Without discussing too much details of renormalization we need to say a few words about it. The interested reader is refered to more details in \cite{[CE]}
or \cite{[MV]}. For $j=0,1,...$ we denote by $\MM^{j}$ the set of those kneading
sequences $\uM$ for which there exists $\bbb \in (({\sqrt 2})^{j+1},{\sqrt 2}^{j}]$
such that $\uM=K(\frac{1}{2},\bbb).$
The kneading sequences in $\MM^{0}$ correspond to the non-renormalizable case.
We denote by $U^{j}$ the set of those $(\aaa,\bbb)\in U$
for which $K(\aaa,\bbb)\in \MM^{j}$. In \cite{[MV]}, $D_{0}$ denotes the region of
those $(\lll,\mmm)\in D$ for which $\lll>\frac{\mmm}{\mmm^{2}-1}$.
This is the non-renormalizable region in the $\lll-\mmm$-parametrization.
In \cite{[BB]} and \cite{[MV]} mainly the non-renormalizable region is considered.
In Section 5 of \cite{[MV]} renormalization, and the way of extension of the result
obtained for the non-renormalizable case is discussed. It turns out that if 
$\cak(\lll,\mmm)\in \MM^{j}$ with $j\geq 1$ then $F^{2}_{\lll,\mmm}$ can be restricted onto a suitable interval mapped into itself by this map. This restriction
is topologically conjugate to $F_{\mmm^{2},\lll\mmm}$ and $\cak(\mmm^{2},\lll\mmm)\in \MM^{j-1}$. In our parametrization if $K(\aaa,\bbb)\in \MM^{j}$
with $j\geq 1$ then $T^{2}_{\aaa,\bbb}$ restricted onto a suitable interval
is topologically conjugate to $T_{1-\aaa,\bbb^{2}/(1-\aaa)}$ and $K(1-\aaa,\bbb^{2}/(1-\aaa))\in \MM^{j-1}$. 
In this paper we only use that the density of Markov maps
in $U^{1}$, shown in \cite{[BB]} implies via renormalization
density of Markov maps in $U$.

%%%%%%%%%%%%%%%%%%%%%%%

We recall a corollary of Theorem C of \cite{[MV]} adapted to our $\aaa-\bbb$-parametrization.
\begin{theorem}\label{MVab}
For each $\uM \in \mathfrak{M}$ there exist two numbers $\alpha_1(\underline{M}) < \alpha_2(\underline{M})$ and a continuous function $\Psi_{\underline{M}}:(\alpha_1(\underline{M}),\alpha_2(\underline{M})) \rightarrow U$ such that for $(\alpha,\beta) \in U$ we have $K(\alpha,\beta)=\underline{M}$ if and only if $\bbb=\Psi_{\underline{M}}(\alpha)$. 
The graphs of the functions $\Psi_{\underline{M}}$ fill up the whole set U. Moreover, $\lim_{\alpha \rightarrow \alpha_1(\underline{M})+} \Psi_{\underline{M}} (\alpha)=1$ if $\underline{M}\seq  RLR^{\infty}$. If $\underline{M}\prec  RLR^{\infty}$ then the curve $(\alpha, \Psi_{\underline{M}}(\alpha))$ converges to a point on the line segment $\{(\alpha,1-\alpha):0<\alpha<\frac{1}{2}\}$
as $\aaa\to\aaa_{1}(\uM)+$.
 If $\underline{M}=RL^{\infty}$ then $\alpha_1(\underline{M})=0, \, \alpha_2(\underline{M})=1$ and $\Psi_{\underline{M}}(\alpha)=1$ for all $\alpha \in (0,1).$
\end{theorem}

For the skew tent map $T_{\alpha,\beta}, \:(\alpha,\beta) \in U$ we define the Frobenius-Perron operator $P_{\alpha,\beta} : L^1[0,1] \rightarrow L^1[0,1]$ by $$P_{\alpha,\beta} f(x)= \sum_{ z \in \{T_{\alpha,\beta} ^{-1} (x) \}} \frac{f(z)}{|T_{\alpha,\beta}'(z)|},$$ which
in a more explicit form is
\begin{equation}\label{FPdef}
P_{\alpha,\beta}f(x)=\frac{\alpha}{\beta} f\left(\frac{\alpha x}{\beta}\right)+\frac{1-\alpha}{\beta} f\left(1-\frac{1-\alpha}{\beta} x \right)\text{ if $0 \leq x \leq \beta$,}
\end{equation}
  and $P_{\alpha,\beta}f(x)=0$ if $x> \beta$.\\
 %We recall for example from Proposition 4.2.4 of \cite{[BG]} the contraction property of Frobenius-Perron operator
 %\begin{equation}\label{FPcontr}
% \|P_{\alpha,\beta} f\|_1 \leq \|f\|_1.
  %\end{equation}
 We also remind to the definition of the variation of a real function $f:[a,b] \rightarrow \mathbb{R}$.
 $$Vf=V_{[a,b]}f=\sup_{\mathcal{P}} \left\{\sum_{k=1}^n|f(x_k)-f(x_{k-1})|\right\}$$ where $\sup$ is taken for all partitions $\mathcal{P}=\{[x_0,x_1], [x_1,x_2], \dots [x_{n-1},x_n]\}$ of $[a,b]$. If $V_{[a,b]}f< +\infty$ then $f$ is of bounded variation, BV on $[a,b]$.
 \begin{definition}\label{MarkovT}
 Suppose $I=[a,b], \: T: I \rightarrow I.$ A partition \\
 $$\mathcal{P} =\{[a_0,a_1], [a_1,a_2], \dots, [a_{n-1},a_n]\}$$
 of $[a,b]$ is Markov for $T$ if for any $i=1, \dots,n$ the transformation $T|_{(a_{i-1},a_i)}$ is a homeomorphism onto the interior of the connected union of some elements of $\mathcal{P}$, that is onto an interval $(a_{j(i)},a_{k(i)})$.
 \end{definition}
Observe that if $T^n_{\alpha,\beta}(\beta)=\alpha$, that is C appears in $K(\alpha,\beta) \in \mathfrak{M}_{< \infty}$ then the partition determined by the points $\{0,\alpha,\beta,T_{\alpha,\beta}(\beta), \dots, T^{n-1}_{\alpha,\beta}(\beta),1\}$ provides a Markov partition.

%%%%%%%%%%%%%%%%%%%%%%
%%%%%%%%%%%%%%%%%%%%%%
\begin{section}{Absolutely continuous invariant measures and densities for skew tent maps}\label{ACIM}

The classical initial paper on the existence of absolutely continuous invariant measures is
\cite{[LasYor]}, however we will follow the more recent monograph \cite{[BG]}.
First we recall some definitions and results from p. 96 of  \cite{[BG]}.
We denote by $\mathcal{T}(I)$ the set of those transformations $ T: I \rightarrow I$ which satisfy the next two properties:
\begin{itemize}
  \item [I.] $T$ is piecewise expanding, that is there exists a partition
      $\mathcal{P}=\{I_i=[a_{i-1},a_i], i=1, \dots,n\}$ of $I$ such that $T|_{I_i}$ is $C^1$ and $|T'(x)|\geq \alpha>1$ for any $i$ and for all $x \in (a_{i-1},a_i).$
  \item [II.] $g(x)=\frac{1}{|T'(x)|}$ is a function of bounded variation, where $T'(x)$ is an appropriately calculated one-sided derivative at the endpoints of $\mathcal{P}$.
  \end{itemize}

For every $n \geq 1$ we define $\mathcal{P}^{(n)}$ as
$$\mathcal{P}^{(n)}=\bigvee_{k=0}^{n-1} T^{-k}(\mathcal{P})=\{I_{i_0}\: \cap \: T^{-1}(I_{i_1}) \: \cap \: \dots\cap\: T^{-n+1}(I_{i_{n-1}}):I_{i_j} \in \mathcal{P}, j=0, \dots,n-1\}.$$ 
One can easily see that if $T \in \mathcal{T}(I)$ then $T^n$ is piecewise expanding on $\mathcal{P}^{(n)}$. 

 Since  $|T'_{\alpha,\beta}(x)|=\frac{\beta}{\alpha}$ on $[0,\alpha]$ and $|T'_{\alpha,\beta}(x)|=\frac{\beta}{1-\alpha}$ on $[\alpha,1]$, 
 for $(\alpha,\beta) \in  U$
 we obtain that $T_{\alpha,\beta} \in \mathcal{T}([0,1])$
 with $\mathcal{P}=\{[0,\alpha],[\alpha,1]\}$.
 
The next theorem is about the existence of absolutely continuous invariant measures, acims and it is Theorem 5.2.1. from \cite{[BG]}.
\begin{theorem}\label{th521}
If $T \in \mathcal{T}(I)$ then it admits an absolutely continuous invariant measure, acim whose density is of bounded variation.
\end{theorem}
In case of skew tent maps this acim is unique. Theorem 8.2.1 of \cite{[BG]} gives an upper bound on the number of distinct ergodic acims for a $T \in \mathcal{T}(I)$.
\begin{theorem}\label{th821}
Let $T \in \mathcal{T}(I)$ be defined on a partition $\mathcal{P}$. Then the number of distinct ergodic acims for $T$ is at most $\#\mathcal{P}-1$.
\end{theorem}
In our case when $(\alpha,\beta) \in U$ and $I_0=[0,1]$ then $\mathcal{P}=\{[0,\alpha],[\alpha,1]\}.$  Since $\#\mathcal{P}=2$ we obtain that for $T_{\alpha,\beta}$ there is only one ergodic acim. Using this and the results about the spectral decomposition of the Frobenius-Perron operator in Chapter 7 of \cite{[BG]} one can see that invariant densities are linear combinations of densities of ergodic acims. Hence in case of our skew tent maps the following Lemma holds:
\begin{lemma}\label{unique}
For every $(\alpha,\beta) \in U$ there is a unique invariant density for $T_{\alpha,\beta}$, and it is the density of the unique ergodic acim.
\end{lemma}
We need the next proposition  which is a variant of Theorems 10.2.1 and 10.3.2  in \cite{[BG]}.
\begin{proposition}\label{1032}
Suppose $(\alpha_n,\beta_n)\in U$ for $n=0,1, \dots, \quad (\alpha_n,\beta_n)\rightarrow (\alpha_0,\beta_0)$ and $\mathcal{P}_n=\{[0,\alpha_n],[\alpha_n,1]\}.$ Suppose that
\begin{equation}\label{partassn}
\begin{split}
   & \forall \: m \geq 1, \exists \: \delta_m>0 \text{ such that if } \\
     & \mathcal{P}_n^{(m)}=\bigvee_{j=0}^{m-1} T^{-j}_{\alpha_n,\beta_n}(\mathcal{P}_n) \text{ then }  \min_{I \in \mathcal{P}_n^{(m)}}\lambda(I) \geq \delta_m >0.
\end{split}
\end{equation}
Then:\\
$(A)$ For any density $f$ of bounded variation there exists a constant $M$ such that for any $n$ and $k=1,2, \dots$
$$V P^k_{\alpha_n,\beta_n} f \leq M.$$
This implies that for any $n$ there is an invariant density $f_n$ of $T_{\alpha_n,\beta_n}$ and the set $\{f_n\}$ is a precompact set in $L^1([0,1],\lambda)$. \\
$(B)$ Moreover, if $f_{n_k} \rightarrow f_0$ in $L^1$ then $f_0$ is an invariant density for $T_{\alpha_0,\beta_0}$.
\end{proposition}
In a similar situation in \cite{[BB]} there is a direct reference to Theorem 10.3.2 of \cite{[BG]} but it seems that after a careful check, this reference is not applicable in the situation of the Markov approximations in \cite{[BB]}, neither in our case.

 Next we discuss what the problem is with the direct application of Theorem 10.3.2 then  we prove Proposition \ref{1032}. \\
The main problem of the direct application in \cite{[BB]} of the theorems from \cite{[BG]} to the case of approximations by skew tent maps is the following. In the assumptions of these theorems given a piecewise expanding transformation $T: I \rightarrow I$, a family $\{T_n\}_{n \geq 1}$ of approximating Markov transformations associated with $T$ is considered. 
Assume $\mathcal{Q}^{(0)}$ denotes the endpoints of intervals belonging to $\mathcal{P}^{(0)}$, where $\mathcal{P}^{(0)}$ is a partition such that $T$ is $C^1$ and expanding on the partition intervals of $\mathcal{P}^{(0)}$.

If one checks in Section 10.3, p. 217 of \cite{[BG]} the definition of the approximating Markov transformations associated with $T$ one can see that there is a sequence of partitions $\mathcal{P}^{(n)}$. It is supposed that the transformations $T_n$ are piecewise expanding and Markov transformations with respect to $\mathcal{P}^{(n)}$.

Moreover, in assumption (a) on p. 217 of \cite{[BG]} it is stated that if $J=[c,d]\in \mathcal{P}^{n}$ and $J \cap \mathcal{Q}^{(0)}=\emptyset$ then $T_n|_J$ is a $C^1$ monotonic function such that
\begin{equation}\label{endp}
T_n(c)=T(c), \quad T_n(d)=T(d)
\end{equation}
Assumption \eqref{endp} is clearly not satisfied if $(\alpha_n,\beta_n)\rightarrow(\alpha_0,\beta_0), \; (\alpha_n,\beta_n)\neq(\alpha_0,\beta_0)$, $T_n=T_{\alpha_n,\beta_n}$, $T=T_{\alpha_0,\beta_0}$ and $\mathcal{P}^{(n)}$ has subintervals $[c,d]$ which do not contain $0,\: \alpha_0$ or 1. This means that contrary to what is claimed by the authors of \cite{[BB]} Theorem 10.3.2 of \cite{[BG]}, cannot be applied directly to the case of Markov approximations they want to use. Our Proposition \ref{1032} can be used in their case as well. Moreover, it is also an advantage of our Proposition \ref{1032} that we do not assume that the approximating skew tent maps are Markov.
\begin{proof}[Proof of Proposition \ref{1032}.]
First we check that assumptions of Theorem 10.2.1 in \cite{[BG]} are satisfied by $T_{\alpha_n,\beta_n}$ and $T_{\alpha_0,\beta_0}$ given in Proposition \ref{1032}. First observe that by $(\alpha_n,\beta_n)\rightarrow (\alpha_0,\beta_0)$ 
\begin{equation}\label{*asskel1}
\text{we can choose ${\bm{\gamma}}>1$ such that $|T'_{\alpha_n,\beta_n}(x)|\geq {\bm{\gamma}}$}
\end{equation}
 for any $x$ where the derivative exists for any $n$,
 this implies condition (1) of Theorem 10.2.1 of \cite{[BG]}. Since $\frac{1}{|T'_{\alpha_n,\beta_n}|}$ is constant on $(0,\alpha_n)$ and $(\alpha_n,1)$, from $(\alpha_n,\beta_n) \rightarrow (\alpha_0,\beta_0)$ it clearly follows that there exists $W>0$ such that $V\left(\frac{1}{T'_{\alpha_n,\beta_n}}\right)\leq W$ for any $n \in \mathbb{N}$. This shows that condition (2) of Theorem 10.2.1 of \cite{[BG]} is also satisfied. Observe that by $(\alpha_n,\beta_n)\rightarrow (\alpha_0,\beta_0)$ the partitions $\mathcal{P}_n$ have the property that we can choose $\delta>0$ such that if $I \in \mathcal{P}_n$ then $T_{\alpha_n,\beta_n}|_I$ is one-to-one, $T_{\alpha_n,\beta_n}(I)$ is an interval and $\min_{I\in\mathcal{P}_n} \lambda(I)> \delta$. This is condition (3) of Theorem 10.2.1 of \cite{[BG]}.

Finally, \eqref{partassn} is assumption (4) of Theorem 10.2.1. Therefore this theorem is applicable to the sequence $T_{\alpha_n,\beta_n}$. This yields that conclusion (A) of our Proposition \ref{1032} holds true.

 The only thing which needs extra proof that in conclusion (B) the function $f_0$, which is the $L^1$ limit of the $T_{\alpha_{n_k},\beta_{n_k}}$ invariant densities $f_{n_k}$, is $P_{\alpha_0,\beta_0}$ invariant.

In an earlier version of our paper we gave a direct detailed proof of this fact. However, as the referee of our paper pointed out this is an immediate consequence of 14. Corollary and
15. Remark of \cite{[Keller]}. 
The assumptions above are quite similar to the ones we need  for them.
Indeed, by choosing a $k$ in \eqref{*asskel1} such that  $\ggg^{k}>3$ one can see that 
 from $(\aaa_{n},\bbb_{n})\to (\aaa_{0},\bbb_{0})$ and $(\aaa_{n},\bbb_{n})\in U$ it follows that there exists a $k$
 such that  $(T_{\aaa_{n},\bbb_{n}}^{k})'\geq 3>2$ for all $n$s.
 We also have $\sup_{n}||(T_{\aaa_{n},\bbb_{n}}^{k})'||_{\oo}<+\oo$
 and all $(T_{\aaa_{n},\bbb_{n}}^{k})'$ satisfy Lipschitz condition with the
 same constant (restricted to their intervals of monotonicity).
 Property (iii) of 15. Remark of \cite{[Keller]}
 states that 
$$\text{ $\inf\{ \lll(I): I$ is a maximal monotonicity interval for some $T_{\aaa_{n},\bbb_{n}}^{k} \}>0$}$$
is a consequence of assumption \eqref{partassn}.
Finally, from $(\aaa_{n},\bbb_{n})\to (\aaa_{0},\bbb_{0})$ it follows that 
$d(T_{\aaa_{n},\bbb_{n}},T_{\aaa_{0},\bbb_{0}})\to 0$ in the Skhorohod type metric given on p. 324 of \cite{[Keller]}, which is defined as follows, (we use the notation $T_{n}$ for $T_{\aaa_{n},\bbb_{n}}$ and $T_{0}$ for $T_{\aaa_{0},\bbb_{0}}$),
$$d(T_{n},T_{0}):=\inf \{ \eee>0:\exists A\sse [0,1],\exists \sss:[0,1]\to[0,1] \text{ s.t. }$$
$$\lll(A)>1-\eee,\ \sss \text{ is a diffeomorphism,}\  T_{n}|_{A}=T_{0}\circ \sss|_{A},\  \text{and}$$
$$\text{for all }x\in A: |\sss(x)-x|<\eee,\ |(1/\sss'(x))-1|<\eee\}.$$
 \end{proof}
The next lemma shows that if $T_{\alpha_0,\beta_0}$ is non-Markov, that is $K(\alpha_0,\beta_0)\in \mathfrak{M}_{\infty}$ then \eqref{partassn} is satisfied.
\begin{lemma}\label{lempart}
Suppose $(\alpha_0,\beta_0) \in  U,\; K(\alpha_0,\beta_0)=\underline{M} \in \mathfrak{M}_{\infty}.$ The sequence $(\alpha_n,\beta_n) \rightarrow (\alpha_0, \beta_0), \; (\alpha_n,\beta_n) \in  U, \; \mathcal{P}_n=\{[0,\alpha_n],[\alpha_n,1]\}, \; n=0,1, \dots$ then \eqref{partassn} is satisfied.

\end{lemma}
\begin{proof}
Since $\underline{M} \in \mathfrak{M}_{\infty}$  we have $T^{k+1}_{\alpha_0,\beta_0}(\alpha_0)=T^k_{\alpha_0,\beta_0}(\beta_0) \neq \alpha_0$ for $k=0,1, \dots .$
This implies that
\begin{equation}\label{ILB20*aa}
T^{k}_{\alpha_0,\beta_0}(\alpha_0)\neq T^{k'}_{\alpha_0,\beta_0}(\alpha_0) \text{ if } k' > k \geq 0.
\end{equation}
Observe that the division points of $\mathcal{P}_n^{(m)}, \; (n=0,1, \dots)$ are $0,1$ and points of the form $T^{-j}_{\alpha_n,\beta_n}(\alpha_n)$ with  $0 \leq j \leq m-1$. Denote the set of division points of $\mathcal{P}_n^{(m)}$ by $\mathcal{Q}_n^{(m)}$.
By \eqref{ILB20*aa} we have
\begin{equation}\label{ILB24*b}
\alpha_0 \notin T^{-j}_{\alpha_0,\beta_0}(\alpha_0) \text{ for any } j=1,2, \dots
\end{equation}
and in general
\begin{equation}\label{ILB24*a}
 T^{-j'}_{\alpha_0,\beta_0}(\alpha_0) \text{ and } T^{-j}_{\alpha_0,\beta_0}(\alpha_0) \text{ are disjoint finite sets for } j' \neq j.
 \end{equation}
 Indeed, if we had for a $j'>j\geq1, \; x \in T^{-j'}_{\alpha_0,\beta_0}(\alpha_0) \:\cap\: T^{-j}_{\alpha_0,\beta_0}(\alpha_0)$ then $$T^{j'-j}_{\alpha_0,\beta_0}(T^{j}_{\alpha_0,\beta_0}(x))=\alpha_0=T^{j}_{\alpha_0,\beta_0}(\alpha_0)$$ and hence $T^{j'-j}_{\alpha_0,\beta_0}(\alpha_0)=\alpha_0,$ which contradicts \eqref{ILB20*aa}. \\
 Denote by $\delta_{0,m}$ the length of the shortest interval in $\mathcal{P}_0^{(m)}$. By using $\alpha_n \rightarrow \alpha_0, \: \beta_n \rightarrow \beta_0$, \eqref{ILB24*b} and \eqref{ILB24*a} we can select $N_m$ such that
 \begin{equation}\label{ILB25*a}
 \disthau (\mathcal{Q}_n^{(m)},\mathcal{Q}_0^{(m)})<\delta_{0,m}/3 \text{ holds for } n \geq N_m.
\end{equation}
This implies that $ \min_{I \in\mathcal{P}_n^{(m)}} \lambda(I) \geq \delta_{0,m}/3>0$ holds for $n \geq N_m$. Since $\min\{\lambda(I):I \in \mathcal{P}_n^{(m)}, \; n \leq N_m\}>0$ we obtain that \eqref{partassn} is satisfied.
\end{proof}
Finally, in this section we make a few remarks about the Lipschitz property of the isentropes. By Theorem A of \cite{[MV]} if $\mu'>\mu$ and $\lambda'>\lambda$ then the topological entropy of $F_{\lambda',\mu'}$ is larger than that of $F_{\lambda,\mu}$. 
Recalling that $\lambda=\frac{\beta}{\alpha}$ and $\mu=\frac{\beta}{1-\alpha}$ we obtain that if the isentrope $\{(\alpha, \Psi_{\underline{M}}(\alpha)):\alpha \in (\alpha_1(\underline{M}),\alpha_2(\underline{M}))\}$ is passing through the point $(\alpha_0,\beta_0)=(\alpha_0,\Psi_{\underline{M}}(\alpha_0))$ then
\begin{equation}\label{ILD1*a}
\frac{\Psi_{\underline{M}}(\alpha)-\Psi_{\underline{M}}(\alpha_0)}{\alpha-\alpha_0} \leq \frac{\beta_0}{\alpha_0} \text{ for } \alpha>\alpha_0
\end{equation}
and
\begin{equation}\label{ILD1*b}
\frac{\Psi_{\underline{M}}(\alpha)-\Psi_{\underline{M}}(\alpha_0)}{\alpha-\alpha_0} \geq -\frac{\beta_0}{1-\alpha_0} \text{ for } \alpha<\alpha_0.
\end{equation}
Now suppose that we selected an interval $[\overline{\alpha}_1,\overline{\alpha}_2] \subset (\alpha_1(\underline{M}),\alpha_2(\underline{M}))$. Then we can choose a constant $\overline{B}>0$ for which
$$\frac{\Psi_{\underline{M}}(\alpha)-\Psi_{\underline{M}}(\alpha_0)}{\alpha-\alpha_0} \leq \overline{B}
\text{ if $\alpha>\alpha_0, \: \alpha, \alpha_0 \in [\overline{\alpha}_1,\overline{\alpha}_2]$,}
$$  and \\
$$\frac{\Psi_{\underline{M}}(\alpha)-\Psi_{\underline{M}}(\alpha_0)}{\alpha-\alpha_0} \geq -\overline{B}\text{ if $\alpha < \alpha_0, \; \alpha, \alpha_0 \in [\overline{\alpha}_1,\overline{\alpha}_2].$}$$ 

This implies that we proved the following:
\begin{proposition}\label{LIP}
Suppose $\underline{M} \in \mathfrak{M}$ and $[\overline{\alpha}_1,\overline{\alpha}_2] \subset(\alpha_1(\underline{M}), \alpha_2(\underline{M})).$ Then there exists a $\overline{B}$ such that
\begin{equation}\label{*Lipb}
\left|\frac{\Psi_{\underline{M}}(\alpha_1)-\Psi_{\underline{M}}(\alpha_2)}{\alpha_1-\alpha_2}\right| \leq \overline{B}
\end{equation}
 if $\alpha_1, \alpha_2 \in [\overline{\alpha}_1,\overline{\alpha}_2]$, that is $\Psi_{\underline{M}}$ is Lipschitz on $[\overline{\alpha}_1,\overline{\alpha}_2]$ and hence is absolutely continuous on $[\overline{\alpha}_1,\overline{\alpha}_2]$, $\Psi_{\underline{M}}'$ exists almost everywhere on $[\overline{\alpha}_1,\overline{\alpha}_2]$ and for any $\alpha_1,\alpha_2 \in [\overline{\alpha}_1,\overline{\alpha}_2], \: \alpha_1<\alpha_2$ we have $\Psi_{\underline{M}}(\alpha_2)-\Psi_{\underline{M}}(\alpha_1)=\int_{\alpha_1}^{\alpha_2} \Psi_{\underline{M}}'(\alpha) d\alpha.$
\end{proposition}

\begin{remark}\label{*remlip}
From \eqref{ILD1*a} and \eqref{ILD1*b} it is also clear that we have a locally uniform 
Lipschitz property of the isentropes. This means that if $(\aaa_{0},\bbb_{0})\in U$
and $[\aaa_{0}-\ddd,\aaa_{0}+\ddd]\xx [\bbb_{0}-\ddd,\bbb_{0}+\ddd]
\sse U$ then one can choose $\overline{B}$ such that for any 
$\aaa_{1},\aaa_{2}\in U$ if $\Psi_{\underline{M}}(\alpha_1),\Psi_{\underline{M}}(\alpha_2)\in [\aaa_{0}-\ddd,\aaa_{0}+\ddd]\xx [\bbb_{0}-\ddd,\bbb_{0}+\ddd]$ then we have \eqref{*Lipb}.
\end{remark}

\end{section}
\section{Isentropes and Lyapunov exponents, the Markov case}\label{secismar}
First we establish Theorem \ref{thlamark} for the Markov case with an additonal differentiability assumption of the isentrope.
\begin{proposition}\label{prlanmark}
Suppose $(\alpha_0,\beta_0) \in  U,\: \underline{M}=K(\alpha_0,\beta_0) \in \mathfrak{M}_{< \infty}$, that is there exists a minimal $n_{\underline{M}}>1$ such that $T^{n_{\underline{M}}}_{\alpha_0,\beta_0}(\beta_0)=\alpha_0$. Assume that $\Lambda=\Lambda_{\alpha_0,\beta_0}$ denotes the Lyapunov exponent of $T_{\alpha_0,\beta_0}$ and $(\alpha,\Psi_{\underline{M}}(\alpha))$ is the isentrope satisfying $\beta_0=\Psi_{\underline{M}}(\alpha_0)$. We also suppose that $\Psi'_{\underline{M}}(\alpha_0)$ exists, that is the isentrope is differentiable at $\alpha_0$.
 Moreover \eqref{IL1*a}, \eqref{IL1*b}  and \eqref{ILE2*a} are satisfied.
\end{proposition}

\begin{proof}
Since $\underline{M} \in \mathfrak{M}_{<\infty}$ we know that $\{T_{\alpha_0,\beta_0}^{n}(\alpha_0):n \in \N \}$ is a finite set
which has $k=n_{\underline{M}}+1$ many elements. We denote this finite set by $c_1<c_2< \dots < c_k$. Then $T^k_{\alpha_0,\beta_0}(\alpha_0)=\alpha_0, \: c_1=T_{\alpha_0,\beta_0}(\beta_0), \: c_k=\beta_0$ and $[c_1,c_k]$ is the dynamical core of the dynamical system $([0,1], T_{\alpha_0,\beta_0})$. The orbit of any $x \in (0,1)$ enters $[c_1,c_k]$ and then for higher iterates $T^n_{\alpha_0,\beta_0}(x)$ stays in this interval. \\
Moreover, since $T_{\alpha_0,\beta_0}([c_1,c_k])=[c_1,c_k]$ we can study the restriction of $T_{\alpha_0,\beta_0}$ onto $[c_1,c_k]$,
which for ease of notation is still denoted by $T_{\alpha_0,\beta_0}$. \\
Since $\mu$ can be obtained as the weak limit of a subsequence of the measures $\frac{1}{N} \sum_{n=0}^{N-1}\bm{\delta}_{T^n_{\alpha_0,\beta_0}(x)}$ for $\mu$ almost every $x$, it is clear that the support of $\mu$ is a subset of $[c_1,c_k]$.
(Recall that $\bm{\ddd}_{x}$ is the Dirac measure centred on $x$.)
 By Proposition \ref{unique}, $\mu$ is unique and ergodic.
By \eqref{IL3*aa}, $\gamma$ in \eqref{IL1*a} satisfies \eqref{ILE2*a} and by Birkhoff's ergodic theorem
\begin{equation}\label{IL3*a}
\gamma= \lim_{N \rightarrow \infty} \frac{1}{N} \sum_{n=0}^{N-1} \chi_{[0,\alpha_0]}(T^n_{\alpha_0,\beta_0}(x))=\mu([0,\aaa_{0}])
\end{equation}
holds for $\mu$ almost every $x$. Since $\mu$ is absolutely continuous with respect to the Lebesgue measure the set $S_{\gamma}$ which consist of those $x$ for which \eqref{IL3*a} holds is of positive Lebesgue measure. It is also well-known, and is easy to check, that the partition $\mathcal{P}_{\alpha_0}=\{[c_1,c_2], \dots, [c_{k-1},c_k]\}$ is a Markov partition of the dynamical core $[c_1,c_k]$. \\
We select $\overline{\alpha}_1<\overline{\alpha}_2$ such that $\alpha_0 \in (\overline{\alpha}_1,\overline{\alpha}_2)\sse [\overline{\alpha}_1,\overline{\alpha}_2]\subset (\alpha_1(\underline{M}),\alpha_2(\underline{M}))$. Since $\Psi_{\underline{M}}$ is an isentrope,
 the maps $T_{\alpha,\Psi_{\underline{M}}(\aaa)}$ are topologically conjugate, 
\begin{equation}\label{IL4*a}
  T^k_{\alpha,\Psi_{\underline{M}}(\alpha)}(\alpha)=\alpha \text{ holds for } \alpha \in [\overline{\alpha}_1, \overline{\alpha}_2],
\end{equation}
and the dynamical systems $T_{\alpha,\Psi_{\underline{M}}(\alpha)}$ are also Markov with Markov partitions $\mathcal{P}_{\alpha}=\{[c_1(\alpha),c_2(\alpha)], \dots, [c_{k-1}(\alpha), c_k(\alpha)]\}$ where $c_i(\alpha)=T^{n_i}_{\alpha,\Psi_{\underline{M}}(\alpha)}(\alpha)$ 
with $n_i < k$ not depending on $\alpha$.
 By Proposition \ref{LIP} and by topological conjugacy of the maps $T_{\alpha,\Psi_{\underline{M}}(\aaa)}, \: \alpha \in [ \overline{\alpha}_1, \overline{\alpha}_2]$ the functions $c_i(\alpha), \: i=1, \dots, k$ are Lipschitz on $[ \overline{\alpha}_1, \overline{\alpha}_2]$.
Moreover, we can choose $M_c>0$ such that
\begin{equation}\label{IL6*a}
  |c_i(\alpha_1)-c_i(\alpha_2)| \leq M_c|\alpha_1-\alpha_2| \text{ for } \alpha_1, \alpha_2 \in [ \overline{\alpha}_1, \overline{\alpha}_2] \text{ and } i=1, \dots,k.
\end{equation}
We denote by $\bm{\varDelta}_c$ the minimum distance among the points $c_i=c_i(\alpha_0), \: i=1, \dots,k$ that is
\begin{equation}\label{IL6*b}
\bm{\varDelta}_c=\min\{c_{i+1}-c_i: i=1, \dots k-1\}.
\end{equation}

Next, proceeding towards a contradiction we suppose that $\gamma$ 
defined in \eqref{IL3*a} does not satisfy \eqref{IL1*b}. By Proposition \ref{LIP}, $\Psi_{\underline{M}}$ is a Lipschitz function on $[ \overline{\alpha}_1, \overline{\alpha}_2]$. Hence $\Psi'_{\underline{M} }(\alpha)$ exists almost everywhere on $[ \overline{\alpha}_1, \overline{\alpha}_2]$ and we can put
\begin{equation}\label{IL7*a}
\widehat{\gamma}(\alpha)=\alpha(1-\alpha) \frac{\Psi'_{\underline{M}}(\alpha)}{\Psi_{\underline{M}}(\alpha)}+\alpha \text{ for $\lll $ a.e. } \alpha \in [\overline{\alpha}_1, \overline{\alpha}_2].
\end{equation}
Since $\Psi_{\underline{M}}(\alpha_0)=\beta_0$ our assumption that $\gamma$ does not satisfy \eqref{IL1*b} can be written in the form $\widehat{\gamma}(\alpha_0)\neq\gamma$. Recall that we supposed that $\Psi'_{\underline{M}}(\alpha_0)$ exists and hence $\widehat{\gamma}(\alpha_0)$ is well defined.
Moreover $\Psi_{\underline{M}}(\alpha_0)=\beta_0$ and \eqref{IL7*a} imply
\begin{equation}\label{psima}
\frac{\widehat{\gamma}(\alpha_0)-\alpha_0}{\alpha_0(1-\alpha_0)}=\frac{\Psi'_{\underline{M}}(\alpha_0)}{\beta_0}, \text{ that is }
0= \frac{\Psi'_{\underline{M}}(\alpha_0)}{\beta_0} -\frac{\widehat{\gamma}(\alpha_0)}{\alpha_0}+\frac{1-\widehat{\gamma}(\alpha_0)}{1-\alpha_0},
\end{equation}
since
$$\frac{\Psi'_{\underline{M}}(\alpha_0)}{\beta_0}-\frac{\widehat{\gamma}(\alpha_0)}{\alpha_0}+\frac{1-\widehat{\gamma}(\alpha_0)}{1-\alpha_0}$$
$$=\frac{\Psi'_{\underline{M}}(\alpha_0)}{\beta_0}+\frac{(1-\widehat{\gamma}(\alpha_0))\alpha_0-\widehat{\gamma}(\alpha_0)(1-\alpha_0)}{\alpha_0(1-\alpha_0)}$$
$$=\frac{\Psi'_{\underline{M}}(\alpha_0)}{\beta_0}+\frac{\alpha_0-\widehat{\gamma}(\alpha_0)}{\alpha_0(1-\alpha_0)}=0.$$

Put $\ds s(\alpha,t)=\Psi_{\underline{M}}(\alpha)\left(\frac{1}{\alpha}\right)^t\left(\frac{1}{1-\alpha}\right)^{1-t}$. Then $\partial_1 s(\alpha,t)$ exists at $\alpha_0$ and for fixed $t, \: s(\alpha,t)$ is Lipschitz in $\alpha$ on $[\overline{\alpha}_1, \overline{\alpha}_2]$.
Using \eqref{psima} we obtain
%\begin{equation}\label{IL13*b}
\begin{align}\nonumber
    \partial_1s(\alpha_0,t)&=\left(\frac{\Psi'_{\underline{M}}(\alpha_0)}{\Psi_{\underline{M}}(\alpha_0)}-\frac{t}{\alpha_0}+\frac{1-t}{1-\alpha_0}\right)s(\alpha_0,t) \\
 \label{IL13*b} %   &=\left(\frac{\Psi'_{\underline{M}}(\alpha_0)}{\beta_0}-\frac{t}{\alpha_0}+\frac{1-t}{1-\alpha_0}\right)s(\alpha_0,t)= \\
     &=\left(\frac{\Psi'_{\underline{M}}(\alpha_0)}{\beta_0}-\frac{\hat{\gamma}(\alpha_0)}{\alpha_0}+\frac{1-\widehat{\gamma}(\alpha_0)}{1-\alpha_0}+\frac{\widehat{\gamma}(\alpha_0)-t}{\alpha_0}-\frac{t-\widehat{\gamma}(\alpha_0)}{1-\alpha_0}\right)s(\alpha_0,t)\\ \nonumber
     &=s(\alpha_0,t)(\widehat{\gamma}(\alpha_0)-t)\left(\frac{1}{\alpha_0}+\frac{1}{1-\alpha_0}\right).
\end{align}
%\end{equation}
Since $\widehat{\gamma}(\alpha_0)-\gamma \neq 0$ we have $\partial_1s(\alpha_0,\gamma)\neq 0$. Select and fix $\delta_0>0$ such that for $|\Delta \alpha|<\delta_0$
\begin{equation}\label{IL14*a}
|s(\alpha_0+\Delta \alpha,\gamma)-s(\alpha_0,\gamma)-\Delta\alpha\cdot \partial_1 s(\alpha_0,\gamma)|<\frac{1}{2}|\Delta\alpha|\cdot |\partial_1s(\alpha_0,\gamma)|.
\end{equation}
Since $s(\alpha_0,\gamma)>0$, by \eqref{IL13*b}, $\sgn(\partial_1s(\alpha_0,\gamma))=\sgn(\widehat{\gamma}(\alpha_0)-\gamma)$. Choose $\Delta\alpha$ with $|\Delta\alpha|<\delta_0$ such that
\begin{equation}\label{ILE601*a}
  \alpha_0+\Delta\alpha \in [\overline{\alpha}_1, \overline{\alpha}_2], \; \Delta\alpha \cdot \partial_1 s(\alpha_0,\gamma)<0, \text{ and } |\Delta\alpha|<\frac{\bm{\varDelta}_c}{4M_c}.
\end{equation}
By \eqref{IL14*a}
\begin{equation}\label{IL14*b}
s(\alpha_0+\Delta \alpha, \gamma)<s(\alpha_0,\gamma)+\frac{1}{2}\Delta\alpha \cdot \partial_1  s(\alpha_0,\gamma)<s(\alpha_0,\gamma).
\end{equation}
Since $s(\alpha_0,t)$ and $\partial_1 s (\alpha_0,t)$ are continuous in  $t$, choose $\delta_1>0$ such that if $|t-\gamma|<\delta_1$ then
\begin{equation}\label{IL14*bb}
s(\alpha_0+\Delta \alpha,t )< s(\alpha_0,t)+\frac{1}{2} \Delta \alpha \cdot
\partial_1 s(\alpha_0,t), \text{ and } |\widehat{\gamma}(\alpha_0)-t|>\frac{|\widehat{\gamma}(\alpha_0)-\gamma|}{2}.
\end{equation}
Put $$\gamma_N(x)=\frac{1}{N}\sum_{n=0}^{N-1} \chi_{[0,\alpha_0]}(T^n_{\alpha_0,\beta_0}(x)).$$
 By Lemma
 \ref{unique}, $\mu$ is ergodic and hence $\gamma_N(x) \rightarrow \gamma=\mu([0,\alpha_0])$ for $\mmm$ a.e. $x$ and there exists $\widehat{S}_{\gamma} \subset S_{\gamma}$ and $N_0 \in \mathbb{N}$ such that $\lambda(\widehat{S}_{\gamma})>\lambda(S_{\gamma})/2>0$ and  we have
\begin{equation}\label{IL8*a}
|\gamma_N(x)-\gamma|< \delta_1 \text{ for any $N \geq N_0$ and } x \in \widehat{S}_{\gamma}.
\end{equation}
We will fix an $N \geq N_0$ later. Suppose $N$ is given and fixed. We can select a system of intervals $I_l=[d_l,e_l]$ such that $T^N_{\alpha_0,\beta_0}$ is linear and non-constant on $I_l$ but is non-linear on any larger interval containing $I_l$, moreover
\begin{equation}\label{IL9*a}
(d_l,e_l) \cap \widehat{S}_{\gamma} \neq \emptyset \text{ and } \widehat{S}_{\gamma} \sse \bigcup_l I_l.
\end{equation}
%%%%%%%%%%%%%%%%%%%%%%%%%%%%%

\begin{figure}[h]
\centering{
\resizebox{1.15\textwidth}{!}{% \usepackage[usenames,dvipsnames]{pstricks}
% \usepackage{epsfig}
% \usepackage{pst-grad} % For gradients
% \usepackage{pst-plot} % For axes
% \usepackage[space]{grffile} % For spaces in paths
% \usepackage{etoolbox} % For spaces in paths
% \makeatletter % For spaces in paths
% \patchcmd\Gread@eps{\@inputcheck#1 }{\@inputcheck"#1"\relax}{}{}
% \makeatother
% 
\psscalebox{1.1 1.2} % Change this value to rescale the drawing.
{\large
\begin{pspicture}(3,-5.6816697)(25.069063,2.6816697)
\psline[linecolor=black, linewidth=0.03](4.4090624,-4.71833)(13.409062,-4.71833)(13.409062,-4.71833)
\psline[linecolor=black, linewidth=0.03](6.6090627,-4.51833)(6.6090627,-4.91833)
\psline[linecolor=black, linewidth=0.03](9.809063,-4.51833)(9.809063,-4.91833)
\psline[linecolor=black, linewidth=0.03](6.0090623,-4.51833)(6.0090623,-4.91833)
\psline[linecolor=black, linewidth=0.03](10.409062,-4.51833)(10.409062,-4.91833)
\psline[linecolor=black, linewidth=0.03](6.6090627,-4.71833)(6.6090627,2.68167)
\psline[linecolor=black, linewidth=0.03](9.809063,-4.71833)(9.809063,2.68167)
\psline[linecolor=black, linewidth=0.03, linestyle=dashed, dash=0.17638889cm 0.10583334cm](6.0090623,-4.71833)(6.0090623,2.68167)
\psline[linecolor=black, linewidth=0.03, linestyle=dashed, dash=0.17638889cm 0.10583334cm](10.409062,-4.71833)(10.409062,2.68167)
\psline[linecolor=black, linewidth=0.04](5.0090623,1.0816699)(6.6090627,-2.7183301)(9.809063,1.68167)(11.209063,-0.9183301)
\psline[linecolor=black, linewidth=0.04, linestyle=dashed, dash=0.17638889cm 0.10583334cm](5.0090623,-0.5183301)(6.0090623,-2.7183301)(10.409062,1.4816699)(11.209063,0.2816699)(11.209063,0.2816699)
\rput[t](6.6090627,-5.11833){$d_l$}
\rput[t](9.809063,-5.23833){$e_l$}
\rput[tr](6.0090623,-5.11833){$d_l(\aaa_0+\Delta\aaa)$}
\rput[tl](10.409062,-5.11833){$e_l(\aaa_0+\Delta\aaa)$}
\psline[linecolor=black, linewidth=0.03](15.409062,-4.71833)(21.809063,-4.71833)
\rput[tl](10.609062,1.8816699){$c_{N,l}^+(\aaa_0+\Delta\aaa)$}
\rput[r](5.6090627,-2.7183301){$c_{N,l}^-(\aaa_0+\Delta\aaa)$}
\rput[r](9.209063,1.68167){$c_{N,l}^+$}
\rput[l](7.0090623,-2.7183301){$c_{N,l}^-$}
\rput[tl](11.209063,-0.9183301){$T_{\aaa_0,\bbb_0}^N$}
\rput[l](11.009063,-0.11833008){$T_{\alpha_0+\Delta \alpha,\Psi_{\underline{M}}(\alpha_0+\Delta \alpha)}^N$}
\psline[linecolor=black, linewidth=0.03](21.009062,-4.71833)(21.009062,2.68167)
\psline[linecolor=black, linewidth=0.03](18.609062,-4.71833)(18.609062,2.68167)
\pscustom[linecolor=black, linewidth=0.03]
{\large\large\large\large
\newpath
\moveto(18.209063,2.0816698)
}

\pscustom[linecolor=black, linewidth=0.03]
{\large\large\large\large
\newpath
\moveto(18.209063,2.0816698)
}
\pscustom[linecolor=black, linewidth=0.04]
{\large\large\large\large
\newpath
\moveto(17.809063,2.2816699)
\lineto(18.409063,2.18167)
\curveto(18.709063,2.13167)(19.359062,1.93167)(19.709063,1.78167)
\curveto(20.059063,1.6316699)(20.559063,1.3316699)(20.709063,1.18167)
\curveto(20.859062,1.03167)(21.159063,0.7316699)(21.309063,0.5816699)
}
\pscustom[linecolor=black, linewidth=0.04]
{\large\large\large\large
\newpath
\moveto(17.809063,1.28167)
\lineto(18.609062,0.9816699)
\curveto(19.009062,0.8316699)(19.659063,0.5316699)(19.909063,0.3816699)
\curveto(20.159063,0.23166992)(20.559063,-0.06833008)(20.709063,-0.21833009)
\curveto(20.859062,-0.3683301)(21.109062,-0.61833006)(21.409063,-0.9183301)
}
\pscustom[linecolor=black, linewidth=0.04]
{\large\large\large\large
\newpath
\moveto(17.809063,0.2816699)
\lineto(18.309063,-0.018330079)
\curveto(18.559063,-0.16833007)(19.009062,-0.4683301)(19.209063,-0.61833006)
\curveto(19.409063,-0.7683301)(19.809063,-1.06833)(20.009062,-1.21833)
\curveto(20.209063,-1.3683301)(20.559063,-1.6683301)(20.709063,-1.81833)
\curveto(20.859062,-1.96833)(21.109062,-2.2183301)(21.409063,-2.51833)
}
\psline[linecolor=black, linewidth=0.03, linestyle=dashed, dash=0.17638889cm 0.10583334cm](19.809063,-4.71833)(19.809063,0.48166993)
\psdots[linecolor=black, dotsize=0.2](19.829063,0.40166992)
\psdots[linecolor=black, dotsize=0.2](18.609062,0.94166994)
\psdots[linecolor=black, dotsize=0.2](21.029062,0.8416699)
\psdots[linecolor=black, dotsize=0.2](21.049063,-2.13833)
\rput[t](18.609062,-5.07833){$\overline{\alpha}_n$}
\rput[t](21.009062,-5.13833){$\aaa_0$}
\rput[t](19.809063,-5.13833){$\aaa_n$}
\rput[bl](21.229063,-1.9183301){$\overline{\beta}_n$}
\rput[bl](21.289062,1.1016699){$\bbb_0$}
\rput[bl](19.749062,0.70166993){$\bbb_n$}
\rput[tr](18.429062,0.9816699){$\widehat{\beta}_n$}
\rput[r](17.809063,2.48167){$\Psi_{\underline{M}}=\Psi_{K(\aaa_0,\bbb_0)}$}
\rput[tr](18.609062,-0.11833008){$\Psi_{K(\alpha_0, \overline{\beta}_n)}$}
\rput[r](17.809063,1.28167){$ \Psi_{\underline{\widehat{M}}_n}= \Psi_{K(\overline{\aaa}_n,{\widehat{\bbb}}_n)}$}
\psdots[linecolor=black, dotsize=0.2](9.809063,1.68167)
\psdots[linecolor=black, dotsize=0.2](6.6090627,-2.7183301)
\psdots[linecolor=black, dotsize=0.2](6.0090623,-2.7183301)
\psdots[linecolor=black, dotsize=0.2](10.409062,1.4816699)
\end{pspicture}
}}}
\caption{Illustration for the proofs of Proposition \ref{prlanmark} and Theorem \ref{thlamark} } \label{*figslypic1}
\end{figure}
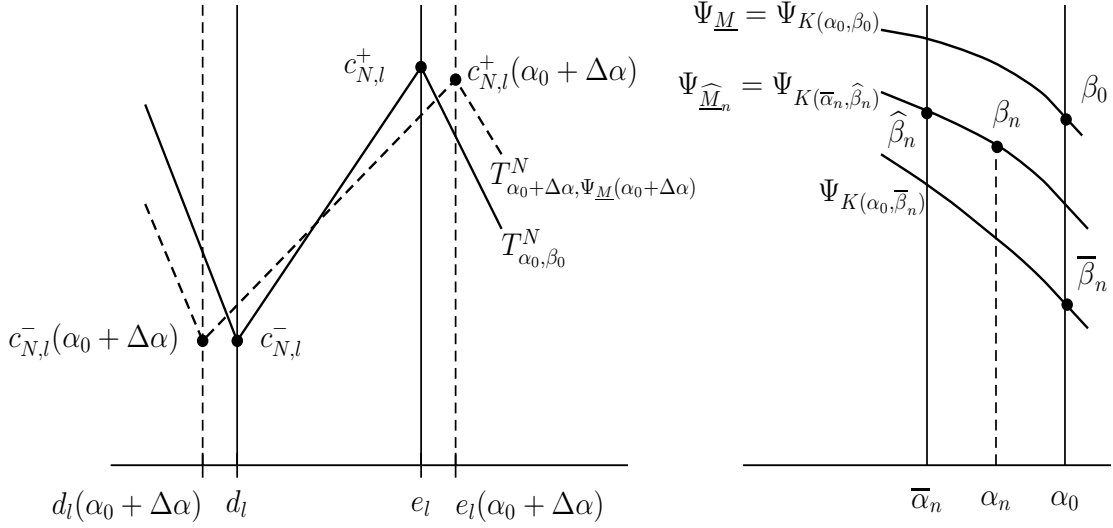

The maximality of the intervals $I_l$ implies that
\begin{equation}\label{IL9*b}
T^N_{\alpha_0,\beta_0}(d_l), T^N_{\alpha_0,\beta_0}(e_l) \in\{c_i:i=1, \dots, k\} \text{ and } T^N_{\alpha_0,\beta_0}(d_l) \neq T^N_{\alpha_0,\beta_0}(e_l).
\end{equation}
From \eqref{IL9*a} it follows that
\begin{equation}\label{IL9*c}
\sum_l \lambda(I_l) \geq \lambda(\widehat{S}_{\gamma}).
\end{equation}
By using \eqref{IL9*b} we introduce the notation
\begin{equation}\label{IL10*aa}
c^-_{N,l}=T^N_{\alpha_0,\beta_0}(d_l) \text{ and } c^+_{N,l}=T^N_{\alpha_0,\beta_0}(e_l).
\end{equation}
From \eqref{IL6*b} and \eqref{IL9*b} it follows that
\begin{equation}\label{IL10*a}
|c_{N,l}^+-c_{N,l}^-| \geq \bm{\varDelta}_c .
\end{equation}
An elementary calculation shows that
\begin{equation}\label{IL10*b}
\begin{split}
   &\left|\frac{d}{dx}(T^N_{\alpha_0,\beta_0}(x))\right|=\left|(T^N_{\alpha_0,\beta_0})'(x)\right|=\left(\left(\frac{\beta_0}{\alpha_0}\right)^{\gamma_N(x)}\left(\frac{\beta_0}{1-\alpha_0}\right)^{1-\gamma_N(x)}\right)^N\\ &\text{ for any } x \in (d_l,e_l).
\end{split}
\end{equation}
During the rest of the proof the reader might find useful to look every so often at
the left half of Figure \ref{*figslypic1}. 
Before getting into more details we try to help the reader by the next heursitic argument.
Looking at the figure one can see one interval of monotonicity for $T^N_{\alpha_0,\beta_0}$. If we change $\aaa_{0}$ to $\aaa_{0}+\Delta\aaa$ then the slope $dT^N_{\alpha_0,\beta_0}(x)/dx$ changes to $dT^N_{\alpha_0+\Delta\aaa,\beta_0}/dx$, the endpoints of these intervals are changing to $e_{l}(\alpha_0+\Delta\aaa)$
and $d_{l}(\alpha_0+\Delta\aaa)$ while the graph of $T^N_{\alpha_0+\Delta\aaa,\beta_0}$ is a segment connecting the points $(d_{l}(\alpha_0+\Delta\aaa),c_{N,l}^{-}(\alpha_0+\Delta\aaa))$ and $(e_{l}(\alpha_0+\Delta\aaa),c_{N,l}^{+}(\alpha_0+\Delta\aaa))$. 
Since we work with the Markov case there is a uniform bound $M_{c}$ from \eqref{IL6*a} (independent of $N$)
on the rate of change of $c_{N,l}^{-}(\alpha_0+\Delta\aaa)$ and $c_{N,l}^{+}(\alpha_0+\Delta\aaa)$. Hence for large $N$s the main effect on the change in the lengths
of the intervals $[e_{l},d_{l}]$ is due to the fact that the slope of the function
$T^N_{\alpha_0,\beta_0}$ is changing over these intervals. We show that if on "many
intervals" $\ggg_{N}(x)$  is not close to the value of $\ggg$ given in \eqref{IL1*b}
then changing $\aaa_{0}$ to a suitably chosen $\aaa_{0}+\DDD\aaa$ will force the absolute value of $dT^N_{\alpha_0+\Delta\aaa,\beta_0}/dx$ decrease significantly
compared to $dT^N_{\alpha_0,\beta_0}(x)/dx$. This will have an effect that the lengths
of the intervals $[d_{l}(\alpha_0+\Delta\aaa),e_{l}(\alpha_0+\Delta\aaa)]$
should increase significantly
compared to that of $[d_{l},e_{l}]$
 and this might mean that these intervals will not fit
anymore into $[0,1]$, which is the contradiction we will obtain in \eqref{*contr}.

Now we return to the details of the proof.
Observe that the value of $\gamma_N(x)$ is constant on $(d_l,e_l)$. Denote this constant by $g_l$. Using \eqref{IL8*a} and \eqref{IL9*a} we obtain
\begin{equation}\label{ILE8*a}
|g_l-\gamma|<\delta_1.
\end{equation}
By topological conjugacy of $T_{\alpha,\Psi_{\underline{M}}(\alpha)}$ and $ T_{\alpha_0,\beta_0}$ if we change $\alpha \in [\overline{\alpha}_1,\overline{\alpha}_2]$ then the system of maximal intervals of monotonicity of $T_{\alpha,\Psi_{\underline{M}}(\alpha)}$
 is not changing in number and only endpoints of these intervals vary in a Lipschitz continuous way. This means that we can consider the intervals $[d_l(\alpha),e_l(\alpha)], \: \alpha \in [\overline{\alpha}_1,\overline{\alpha}_2]$ and the absolute value of the slope of $T^N_{\alpha,\Psi_{\underline{M}}(\alpha)}$ on these intervals will be for any $x \in(d_l(\alpha),e_l(\alpha))$
\begin{equation}\label{IL11*a}
\begin{split}
&\left|\frac{d}{dx}(T^N_{\alpha,\Psi_{\underline{M}}(\alpha)}(x))\right|=\left(\left(\frac{\Psi_{\underline{M}}(\alpha)}{\alpha}\right)^{g_l} \cdot\left(\frac{\Psi_{\underline{M}}(\alpha)}{1-\alpha}\right)^{1-g_l}\right)^N\\
&=(\Psi_{\underline{M}}(\alpha))^N \cdot \left(\left(\frac{1}{\alpha}\right)^{g_l} \cdot\left(\frac{1}{1-\alpha}\right)^{1-g_l}\right)^N=(s(\alpha,g_l))^N.
\end{split}
\end{equation}
By \eqref{IL14*bb} and \eqref{IL8*a} we have
\begin{equation}\label{IL14*bbb}
  s(\alpha_0+\Delta \alpha, g_l)<s(\alpha_0,g_l)+\frac{1}{2} \Delta \alpha \partial_1 s(\alpha_0,g_l),
\end{equation}
that is
\begin{equation}\label{IL14*c}
\begin{split}
  &\frac{s(\alpha_0,g_l)}{s(\alpha_0+\Delta \alpha,g_l)}>\frac{1}{1+\frac{1}{2}  \Delta \alpha \frac{\partial_1 s(\alpha_0,g_l)}{s(\alpha_0,g_l)}}\\
  &=\frac{1}{1-\frac{1}{2}\left|\frac{\partial_1s(\alpha_0,g_l)}{s(\alpha_0,g_l)}\right||\Delta\alpha|}>1+\frac{1}{2} \left|\frac{\partial_1s(\alpha_0,g_l)}{s(\alpha_0,g_l)}\right|\cdot|\Delta\alpha|.
\end{split}
\end{equation}
Using $\frac{1}{\alpha}+\frac{1}{1-\alpha} \geq 2$, \eqref{IL13*b}, \eqref{IL14*bb}, \eqref{IL14*c} and Bernoulli's inequality
\begin{equation}\label{IL15*a}
\begin{split}
&(s(\alpha_0+\Delta \alpha,g_l))^N<\frac{(s(\alpha_0,g_l))^N}{\left(1+\frac{1}{2}\left|\frac{\partial_1s(\alpha_0,g_l)}{s(\alpha_0,g_l)}\right||\Delta \alpha|\right)^N}\\
&<\frac{(s(\alpha_0,g_l))^N}{1+N \cdot \frac{1}{4}|\widehat{\gamma}(\alpha_0)-\gamma|(\frac{1}{\alpha}+\frac{1}{1-\alpha})|\Delta\alpha|}<\frac{(s(\alpha_0,g_l))^N}{1+ \frac{N}{2}|\widehat{\gamma}(\alpha_0)-\gamma||\Delta\alpha|}.
\end{split}
\end{equation}
Since the choice of $\Delta \alpha$ did not depend on $N$ we can suppose that $N$ is so large that
\begin{equation}\label{IL15*b}
1+\frac{N}{2} |\widehat{\gamma}(\alpha_0)-\gamma|\cdot  |\Delta \alpha | > \frac{20}{\lambda(\widehat{S}_{\ggg})}.
\end{equation}
By \eqref{IL10*b} and \eqref{IL11*a} we know that
\begin{equation}\label{IL16*b}
  \lambda(I_l)=e_l-d_l=\frac{|c^+_{N,l}-c^-_{N,l}|}{(s(\alpha_0,g_l))^N}=\frac{|c^+_{N,l}(\alpha_0)-c^-_{N,l}(\alpha_0)|}{(s(\alpha_0,g_l))^N}.
\end{equation}
We want to obtain an estimate of $e_l(\alpha_0+\Delta \alpha)-d_l(\alpha_0+\Delta \alpha)$. By \eqref{IL6*a} $$|c^{\pm}_{N,l}(\alpha_0+\Delta \alpha)-c^{\pm}_{N,l}(\alpha_0)|\leq M_c \cdot |\Delta \alpha|,$$ and hence using \eqref{IL6*b} and \eqref{ILE601*a}
\begin{equation}\label{IL16*a}
\begin{split}
  &|c^+_{N,l}(\alpha_0+\Delta \alpha)-c^-_{N,l}(\alpha_0+\Delta\alpha)|\\
  &>|c^+_{N,l}(\alpha_0)-c^-_{N,l}(\alpha_0)|-2M_c|\Delta\alpha|>\frac{1}{2}|c^+_{N,l}(\alpha_0)-c_{N,l}^-(\alpha_0)|.
\end{split}
\end{equation}
By \eqref{IL11*a}, \eqref{IL15*a}, \eqref{IL15*b}, \eqref{IL16*b} and \eqref{IL16*a} we obtain
\begin{equation}\label{IL17*a}
\begin{split}
 &e_l(\alpha_0+\Delta \alpha)-d_l(\alpha_0+\Delta \alpha)=\frac{|c^+_{N,l}(\alpha_0+\Delta \alpha)-c^-_{N,l}(\alpha_0+\Delta \alpha)|}{|s(\alpha_0+\Delta \alpha,g_l)|^N}\\
 &>\frac{\frac{1}{2}|c^+_{N,l}(\alpha_0)-c^-_{N,l}(\alpha_0)|}{|s(\alpha_0,g_l)|^N} \cdot
 \Big (1+\frac{N}{2}|\widehat{\gamma}(\alpha_0)-\gamma|\cdot|\Delta\alpha|\Big )\\
 &>\frac{|c^+_{N,l}(\alpha_0)-c^-_{N,l}(\alpha_0)|}{|s(\alpha_0,g_l)|^N}\cdot\frac{10}{\lambda(\widehat{S}_{\gamma})}=\lambda(I_l) \cdot \frac{10}{\lambda(\widehat{S}_\gamma)}.
\end{split}
\end{equation}
By topological conjugacy of $T_{\alpha_0+\Delta \alpha,\Psi_{\underline{M}}(\alpha_0+\Delta \alpha)}$ and $T_{\alpha_0,\beta_0}$ the intervals $I_l(\alpha_0+\Delta \alpha)=[d_l(\alpha_0+\Delta \alpha),e_l(\alpha_0+\Delta \alpha)]$ are non-overlapping for fixed $\Delta\alpha$ and are in $[0,1]$. This contradicts \eqref{IL9*c} since we have 
\begin{equation}
\label{*contr}
1 \geq \sum_{l}\lambda(I_l(\alpha_0+\Delta\alpha))>\sum_l \lambda(I_l)\cdot \frac{10}{\lambda(\widehat{S}_{\gamma})}\geq 10.
\end{equation}
 Hence $\gamma$ satisfies \eqref{IL1*b} and Proposition \ref{prlanmark} is proved.
\end{proof}

%%%%%%%%%%%%%%%%%%%%%%%%%%%
\section{Differentiability of the isentropes (ergodic theory approach)}\label{Disentr}

In this section we prove that isentropes are continuously differentiable
curves. We have already seen that results of \cite{[MV]} imply that they are
(locally uniformly) Lipschitz. There are two possible ways to verify
that they are differentiable. One way, the one which we call analytic method,
is to use the auxiliary function $\TTT_{\uM}$,
\eqref{*implder}
 and implicit differentiation. If one can verify that for  $(\aaa,\bbb)\in U$, $\uM=K(\aaa,\bbb)$ we have 
$\dd_{2}\TTT_{\uM}(\aaa,\bbb)\not= 0$ then this argument works. Unfortunately,
to deal with partial derivatives of $\TTT_{\uM}$ is a quite unpleasant and technical task.  We have a manuscript in prepartion, \cite{[BKtheta]} which discusses this other approach. 
In this paper we use a much more elegant and less technical argument which 
we called the ergodic theory approach and
is based on  Proposition \ref{prlanmark} which says that
the slope of the tangent of isentropes wherever it exists can be expressed by $\ggg$, which depends
on the unique acim of the skew tent map considered.
Then by using approximations, Proposition \ref{1032} and uniqueness of the acim first we verify in Lemma \ref{thdiffmark} continuous differentiability of the isentrope
in the Markov case. Then by another approximation argument we prove the general case in Theorem \ref{thdiffnonmark}.

\begin{lemma}\label{thdiffmark}
If $\underline{M} \in \mathfrak{M}_{< \infty}$ then $\Psi'_{\underline{M}}$ exists and is continuous on $(\alpha_1(\underline{M}),\alpha_2(\underline{M}))$.
\end{lemma}

\begin{proof}
Choose $\alpha_0 \in(\alpha_1(\underline{M}),\alpha_2(\underline{M}))$. We know that $\Psi'_{\underline{M}}(\aaa)$ exists for almost every  $\aaa\in (\alpha_1(\underline{M})),\alpha_2(\underline{M}))$.
 Denote by $D_{\underline{M}}$ the set of those $\alpha$s where $\Psi'_{\underline{M}}$ exists. Suppose that there exists $d_1 \neq d_2 \in [-\infty,\infty]$ and $\alpha_{i,n} \rightarrow \alpha_0, \: (i=1,2)$ such that $\alpha_{i,n} \in D_{\underline{M}}$, and $\Psi'_{\underline{M}}(\alpha_{i,n})\rightarrow d_i, \: (i=1,2)$. Put $\beta_{i,n}=\Psi_{\underline{M}}(\alpha_{i,n}), \: i=1,2$. Then $(\alpha_{i,n},\beta_{i,n}) \rightarrow (\alpha_0,\beta_0)=(\alpha_0, \Psi_{\underline{M}}(\alpha_0))$, for $i=1,2$ as $n \rightarrow \infty$.
 Since $\Psi_{\underline{M}}$ is an isentrope we know that the maps $T_{\alpha_{i,n},\beta_{i,n}}$ are all topologically conjugate to $T_{\alpha_0,\beta_0}$. It is not difficult to check that the assumptions of Proposition \ref{1032} are satisfied. Hence if we denote by $f_{i,n}$ the invariant densities of $T_{\alpha_{i,n},\beta_{i,n}}$ which appear in Proposition \ref{1032} then there are subsequences $n_{k,i}$ such that $f_{i, n_{k,i}} \rightarrow f_{i,0}$ in $L^1$, and $f_{i,0}, \: (i=1,2)$ are both invariant densities of $T_{\alpha_0,\beta_0}$.
 By Proposition \ref{unique}, $T_{\alpha_0,\beta_0}$ has a unique invariant density and hence $f_{1,0}=f_{2,0}=f_0$ almost everywhere. Denote by $\mu_{i,n}$ and $\mu_0$ the acims with densities $f_{i,n}$ and $f_0$, respectively. For $i=1,2$ we have
\begin{equation}\label{ILB7*a}
\gamma_{i,n_{k,i}}=\alpha_{i,n_{k,i}}(1-\alpha_{i,n_{k,i}})\frac{\Psi'_{\underline{M}}(\alpha_{i,n_{k,i}})}{\Psi_{\underline{M}}(\alpha_{i,n_{k,i}})}+\alpha_{i,n_{k,i}} \rightarrow \gamma_i=\alpha_0(1-\alpha_0)\frac{d_i}{\beta_0}+\alpha_0.
\end{equation}
From $d_1 \neq d_2$ it follows that $\gamma_1 \neq \gamma_2$. By Proposition \ref{prlanmark}
\begin{equation}\label{ILI*3b}
\gamma_{i,n_{k,i}}=\mu_{n_{k,i}}([0,\alpha_{i,n_{k,i}}])=\int_{[0,\alpha_{i,n}]}f_{i,n_{k,i}} d\lambda, \: i=1,2.
\end{equation}
Set $\gamma_0=\mu_0([0,\alpha_0])=\int_{[0,\alpha_0]}f_0 d\lambda$. We denote by $I_{k,i}$ the interval with endpoints $\alpha_0$ and $\alpha_{i,n_{k,i}}$. We know that \begin{equation}\label{ILI*3a}
  \int_{[0,1]} |f_{i,n_{k,i}}-f_0|d \lambda \rightarrow 0 \text{ as } k \rightarrow + \infty, \text{ for $i=1,2$}.
\end{equation}
Hence
\begin{equation}\label{ILI*4a}
\begin{split}
&|\gamma_0-\gamma_{i,n}|=\left|\int_{[0,\alpha_0]}f_0 d \lambda-\int_{[0, \alpha_{i,n_{k,i}}]}f_{i,n_{k,i}}d \lambda \right|\\
& \leq \left|\int_{I_{k,i}}f_0 d \lambda\right|+\int_{[0, \alpha_{i,n_{k,i}}]}|f_0-f_{i,n_{k,i}}|d \lambda 
\\
 &\leq\left|\int_{I_{k,i}}f_0 d \lambda\right|+\|f-f_{i,n_{k,i}}\|_1 \rightarrow 0 \text{ as } k \rightarrow \infty.
 \end{split}
\end{equation}
Since $\gamma_{i,n_{k,i}} \rightarrow \gamma_i, \: i=1,2$ and $\gamma_1 \neq \gamma_2$, it is impossible that $\gamma_{i,n_{k,i}} \rightarrow\gamma_0, \: i=1,2$. Hence $\Psi'_{\underline{M}}|_{D_{\underline{M}}}$ has a limit 
at every $\alpha_0 \in (\alpha_1(\underline{M}),\alpha_2{\underline{M}})$. Since $\Psi_{\underline{M}}$ is locally Lipschitz and $D_{\underline{M}}$ is of full measure in $(\alpha_1(\underline{M}),\alpha_2(\underline{M}))$ we obtained that $\Psi'_{\underline{M}}(\alpha_0)$ exists 
and $\Psi'_{\underline{M}}$ is  continuous at any 
$\alpha_0 \in (\alpha_1(\underline{M}),\alpha_2(\underline{M}))$.
\end{proof} 

Next we prove Theorem \ref{thdiffnonmark} about the general case.
%\begin{theorem}\label{thdiffnonmark}
%If $\underline{M} \in \mathfrak{M}$ then $\Psi'_{\underline{M}}$ exists and is continuous on $(\alpha_1(\underline{M}),\alpha_2(\underline{M}))$.
%\end{theorem}
\begin{proof}[Proof of Theorem \ref{thdiffnonmark}]
%\begin{proof}
The Markov case $\underline{M} \in \mathfrak{M}_{< \infty}$ is Lemma \ref{thdiffmark}. 
In \cite{[BB]} there are some considerations showing that the curves $\{(\alpha, \Psi_{\underline{M}}(\alpha)): \underline{M}\in \mathfrak{M}^0_{< \infty}\}$ are dense in $ U^0$. By renormalization, or by using directly the argument from \cite{[BB]} one can see that the curves $\{(\alpha,\Psi_{\underline{M}}(\alpha)): \:\underline{M} \in \mathfrak{M}_{< \infty}\}$ are dense in $ U$. Suppose that $\underline{M} \in \mathfrak{M}_{\infty}$ is fixed $\beta_0=\Psi_{\underline{M}}(\alpha_0), \: (\alpha_0, \beta_0) \in  U, \: K(\alpha_0, \beta_0)=\underline{M}.$
 Then there are no $C$s in $\underline{M}$ and $T^{k+1}_{\alpha_0,\beta_0}(\alpha_0)=T^k_{\alpha_0,\beta_0}(\beta_0) \neq \alpha_0$ for any $k \geq0$. This also implies that
\begin{equation}\label{ILB20*a}
T^k_{\alpha_0,\beta_0}(\alpha_0) \neq T^{k'}_{\alpha_0,\beta_0}(\alpha_0) \text{ if } k'>k \geq 0.
\end{equation}
Choose $[\overline{\alpha}_1,\overline{\alpha}_2]\subset (\alpha_1(\underline{M}),\alpha_2(\underline{M}))$. By Proposition \ref{LIP} $,
\Psi_{\underline{M}}$ is Lipschitz on $[\overline{\alpha}_1,\overline{\alpha}_2]$ and $\Psi'_{\underline{M}}$ exists and is bounded almost everywhere on $[\overline{\alpha}_1,\overline{\alpha}_2]$. Suppose that $\Psi_{\underline{M}}$ is not differentiable at $\alpha_0 \in (\overline{\alpha}_1,\overline{\alpha}_2).$ This means that there is $d_1 \neq d_2$ such that we can select $\alpha_{i,n} \rightarrow\alpha_0, \: i=1,2$, such that
\begin{equation}
\frac{\Psi_{\underline{M}}(\alpha_{i,n})-\Psi_{\underline{M}}(\alpha_0)}{\alpha_{i,n}-\alpha_0} \rightarrow d_i, \: i=1,2.
\end{equation}
Since the Markov isentropes are dense in $ U$ we can choose $\underline{M}_n \in \mathfrak{M}_{< \infty}$ such that
\begin{equation}
\begin{split}
 &\frac{\Psi_{\underline{M}_n}(\alpha_{i,n})-\Psi_{\underline{M}_n}(\alpha_0)}{\alpha_{i,n}-\alpha_0} \rightarrow d_i,   \Psi_{\underline{M}_n}(\alpha_{i,n})\rightarrow \beta_0, \\
  &\text{ and } \Psi_{\underline{M}_n}(\alpha_{0}) \rightarrow \beta_0, \text{ as } n \rightarrow \infty, i= 1,2.
  \end{split}
\end{equation}
By Lemma \ref{thdiffmark} and by the Mean Value Theorem we can choose $\overline{\alpha}_{i,n} \rightarrow \alpha_0$ such that
\begin{equation}
  \Psi_{\underline{M}_n}(\overline{\alpha}_{i,n})=\overline{\beta}_{i,n} \rightarrow \beta_0 \text { and } \Psi'_{\underline{M}_n}(\overline{\alpha}_{i,n}) \rightarrow d_i \text{ for } i=1,2.
\end{equation}
We denote by $\mu_{i,n}$ the acim of $T_{\overline{\alpha}_{i,n}, \overline{\beta}_{i,n}}, \: i=1,2$ and $f_{i,n}$ denotes the corresponding invariant density.
 By Lemma \ref{lempart} assumption 
\eqref{partassn} is satisfied for $(\overline{\alpha}_{i,n},\overline{\beta}_{i,n}) \rightarrow (\alpha_0,\beta_0)$
 for $i=1,2$.
 Then we can apply Proposition \ref{1032} in this case as well and we conclude that for suitable subsequences $f_{i,n_{k,i}} \rightarrow f_0$ as $k \rightarrow +\infty$ where $f_0$ is the unique invariant density of $T_{\alpha_0, \beta_0}$. Now by using $\overline{\alpha}_{i,n_{k,i}}$ instead of $\alpha_{i,n_{k,i}}$ one can argue as we did in the end of the proof of Lemma \ref{thdiffmark} to obtain \eqref{ILB7*a}, \eqref{ILI*3b}, \eqref{ILI*3a} and \eqref{ILI*4a}. This way we can obtain a contradiction as in Lemma \ref{thdiffmark}.
\end{proof}

%%%%%%%%%%%%%%%%%%%%%%

\section{Isentropes and Lyapunov exponents, the general case}\label{secisnonmar}
Next we prove Theorem \ref{thlamark}, the main
 result of our paper. Its special Markov case, assuming differentiability of the isentrope at the point considered was discussed in Section \ref{secismar}.
%\begin{theorem}\label{thlamark}
%Suppose $(\alpha_0,\beta_0) \in  U$, 
 %$\Lambda = \Lambda_{\alpha_0,\beta_0}$ denotes the Lyapunov exponent of $T_{\alpha_0,\beta_0}$ and $(\alpha, \Psi_{\underline{M}}(\alpha))$ is the isentrope satisfying $\beta_0=\Psi_{\underline{M}}(\alpha_0)$.
% Then $\Psi'_{\underline{M}}(\alpha_0)$ exists, moreover \eqref{IL1*a} and \eqref{IL1*b} are satisfied.
%\end{theorem}
\begin{proof}[Proof of Theorem \ref{thlamark}]
The case $K(\alpha_0,\beta_0)=\underline{M} \in \mathfrak{M}_{< \infty}$ was proved in Proposition \ref{prlanmark}. By Theorem \ref{thdiffnonmark} we know that $\Psi'_{\underline{M}}(\aaa)$ exists for any $\underline{M} \in \mathfrak{M}$ and $\aaa\in (\aaa_{1}(\uM),\aaa_{2}(\uM))$. Next we suppose that $K(\alpha_0,\beta_0) \in \mathfrak{M}_{\infty}$, that is there is no $C$ in $K(\alpha_0,\beta_0)$. We use again the fact that isentropes corresponding to Markov systems are dense in $ U$. We will select a suitable $(\alpha_n,\beta_n)\rightarrow (\alpha_0,\beta_0)$ such that $K(\alpha_n,\beta_n)=\underline{M}_n \in \mathfrak{M}_{< \infty}$. Again we choose $\overline{\alpha}_1<\overline{\alpha}_2$ such that $\alpha_0 \in (\overline{\alpha}_1,\overline{\alpha}_2)\subset [\overline{\alpha}_1,\overline{\alpha}_2]\subset (\overline{\alpha}_1(\underline{M}),\overline{\alpha}_2(\underline{M}))$. Suppose $n \in \mathbb{N}$ is given. Choose $\overline{\alpha}_n < \alpha_0$ such that
\begin{equation}\label{ILF2*a}
|\overline{\alpha}_n-\alpha_0|< \frac{1}{n} \text{ and } \left|\frac{\Psi_{\underline{M}}(\overline{\alpha}_n)-\Psi_{\underline{M}}(\alpha_0)}{\overline{\alpha}_n-\alpha_0}-\Psi'_{\underline{M}}(\alpha_0)\right|<\frac{1}{2n}.
\end{equation}
Select $\overline{\beta}_n$ such that
\begin{equation}\label{ILF2*b}
0<\bbb_{0}-\overline{\beta}_n=\Psi_{\underline{M}}(\alpha_0)-\overline{\beta}_n<\frac{1}{4n}|\overline{\alpha}_n-\alpha_0|.
 %\text{ and } K(\alpha_0, \overline{\beta}_n)\in \mathfrak{M}_{<\infty}.
\end{equation}
The right half of Figure \ref{*figslypic1} might turn out to be useful 
to help to understand the rest of the proof.
 
Since isentropes do not cross $\Psi_{K(\alpha_0, \overline{\beta}_n)}<\Psi_{K(\alpha_0, {\beta}_0)}=\Psi_{\uM}$ at points where they are both defined.
By choosing $\overline{\beta}_n$ sufficiently close to
$\bbb_{0}$ we can ensure that they are both defined on 
$[\overline{\alpha}_n,{\alpha}_0]. $

Select $\widehat{\beta}_n$ such that
\begin{equation}\label{ILF3*a}
\begin{split}
&0<\Psi_{\uM} (\overline{\alpha}_n)-\widehat{\beta}_n< \frac{1}{4n}|\overline{\alpha}_n-\alpha_0|,\quad \Psi_{K(\alpha_0, \overline{\beta}_n)}(\overline{\alpha}_n)<\widehat{\beta}_n\\
 &\text{ and } K(\overline{\alpha}_n, \widehat{\beta}_n)=\widehat{\underline{M}}_n \in \mathfrak{M}_{< \infty}.
\end{split}
\end{equation}

Since isentropes do not cross we have
\begin{equation}\label{ILF3*b}
  \overline{\beta}_n< \Psi_{\underline{\widehat{M}}_n}(\alpha_0)<\Psi_{\underline{M}_n}(\alpha_0)=\beta_0.
\end{equation}

Recalling that $\Psi_{\underline{\widehat{M}}_n}(\overline{\alpha}_n)=\widehat{\beta}_n$ by \eqref{ILF2*a}, \eqref{ILF2*b}, \eqref{ILF3*a} and \eqref{ILF3*b} we obtain that
\begin{equation}\label{ILF3*c}
\left|\frac{\Psi_{{\widehat{\uM}_n}}(\overline{\alpha}_n)-\Psi_{{\widehat{\uM}_n}}(\alpha_0)}{\overline{\alpha}_n-\alpha_0} - \Psi_{\underline{M}}'(\alpha_0)\right|<\frac{1}{n}.
\end{equation}
Since $\Psi_{\underline{\widehat{M}}_n}$ is differentiable on $[ \overline{\alpha}_n, \alpha_0]$ by the Mean Value Theorem we can choose $\alpha_n \in ( \overline{\alpha}_n, \alpha_0)$ such that
$$\Psi'_{\underline{\widehat{M}}_n}(\alpha_n)=\frac{\Psi_{\underline{\widehat{M}}_n}(\overline{\alpha}_n)-\Psi_{\underline{\widehat{M}}_n}(\alpha_0)}{\overline{\alpha}_n-\alpha_0}. $$ 

From \eqref{ILF3*c} it follows that
\begin{equation}\label{ILF4*a}
|\Psi'_{\underline{\widehat{M}}_n}({\alpha}_n)-\Psi'_{\underline{M}}(\alpha_0)|< \frac{1}{n}.
\end{equation}
Set $\beta_n=\Psi_{\underline{\widehat{M}}_n}(\overline{\alpha}_n)$. 
By the local uniform Lipschitz property of the isentropes mentioned
in Remark \ref{*remlip}
it is clear that $(\alpha_n,\beta_n) \rightarrow(\alpha_0, \beta_0)$. Since $\underline{\widehat{M}}_n \in \mathfrak{M}_{<\infty}$ we can apply Proposition \ref{prlanmark} at the point $(\alpha_n,\beta_n)$ to the isentrope $\Psi_{\underline{\widehat{M}}_n}$. 

By Lemma \ref{lempart}  assumption \eqref{partassn} is satisfied. Hence if $\mu_n$ and $f_n$ denote the acim and its density for $T_{\alpha_n,\beta_n}, \: n=0,1, \dots$ then by Proposition \ref{1032} for a suitable subsequence $n_k$ the sequence $f_{n_k} \rightarrow f_0$ in $L^1$. Now $$\gamma_{n_k}=\mu_{n_k}([0,\alpha_{n_k}])=\int_{0}^{\alpha_{n_k}}f_{n_k}d \lambda$$ $$=
\int_{0}^{\alpha_{0}}f_{0}\: d \lambda+\int_{0}^{\alpha_{n_k}}f_{n_k}\:d \lambda-\int_{0}^{\alpha_{0}}f_0 \:d \lambda=\mu_0([0,\alpha_0])+A_k=\gamma_0+A_k.$$

We have
$$|A_k|=\left |\int_{0}^{\alpha_{n_k}}f_{n_k}d \lambda-\int_{0}^{\alpha_{0}}f_{0}\: d \lambda \right | \leq \int_{0}^{\alpha_{n_k}}|f_{n_k}-f_0|d \lambda+\int_{\alpha_{n_k}}^{\alpha_{0}}f_{0}\: d \lambda$$ $$ \leq \|f_{n_k}-f_0\|_1+\int_{\alpha_{n_k}}^{\alpha_{0}}f_{0}\: d \lambda \rightarrow 0.$$ 
Hence $\gamma_{n_k} \rightarrow \gamma_0$. 

By Proposition \ref{prlanmark} we have
\begin{equation}\label{ILF6*a}
  \Lambda_{\alpha_n,\beta_n}=\Lambda_n=\gamma_n \log \frac{\beta_n}{\alpha_n} + (1-\gamma_n) \log \frac{\beta_n}{1-\alpha_n} \text{ where } \gamma_n \text{ satisfies }
\end{equation}
\begin{equation}\label{ILF6*b}
\gamma_n=\mu_n([0,\alpha_n])=%
\alpha_n(1-\alpha_n)\frac{\Psi'_{\underline{\widehat{M}}_n}(\alpha_n)}{\beta_n}+\alpha_n.
\end{equation}
Using  \eqref{ILF4*a} and $\gamma_{n_k} \rightarrow \gamma_0$ by taking limit as $k \rightarrow \infty$ we obtain that \eqref{IL1*a} and \eqref{IL1*b} hold for $T_{\alpha_0,\beta_0}$.
\end{proof}

%\end{section}
%\begin{cases}
%\frac{\beta}{\alpha}x  \quad \quad \quad \quad \text{if} %\quad   0 \leq x \leq \alpha,\\
%\frac{\beta}{1-\alpha}(1-x)    \quad  \text{if} \quad  %\alpha< x \leq 1.
%\end{cases}

We thank the referee of this paper for making valuable comments and pointing out some useful references.

\end{document}